%% file: main.tex
\documentclass[reqno, 11pt]{amsart}
\input{preambule}

\usepackage{xcolor}

\title[Operations on valuations and constructible functions]{Operations on constructible functions and generalized valuations}
\author{Andreas Bernig}
\address{Institut f\"ur Mathematik, Goethe-Universit\"at Frankfurt, Robert-Mayer-Str. 10, 60054 Frankfurt, Germany}
\email{bernig@math.uni-frankfurt.de}
\author{Vadim Lebovici}
\address{Sorbonne Université, Université Paris Cité, CNRS, IMJ-PRG, F-75005 Paris, France}
\email{lebovici@imj-prg.fr}

\thanks{VL was supported in part by FSMP and in part by EPSRC EP/R018472/1. AB was supported by DFG grant BE 2484/10-1}

\begin{document}

\begin{abstract}
    Alesker's theory of generalized valuations unifies smooth measures and constructible functions on real analytic manifolds, extending classical operations on functions and measures. Alesker showed that these operations agree with the sheaf-theoretic ones on constructible functions under restrictive assumptions, leaving key aspects conjectural. In this article, we close this gap by proving that the two approaches indeed coincide on constructible functions under mild transversality assumptions. Our proof is based on a comparison with the corresponding operations on characteristic cycles. As applications, we extend additive kinematic formulas from convex bodies to compact subanalytic sets in Euclidean spaces and derive new kinematic formulas on the 3-sphere.
\end{abstract}

\maketitle

\tableofcontents

\renewcommand{\thethm}{\arabic{thm}}

\section{Introduction}
\subsection{Context}
A valuation is a finitely additive measure defined on sufficiently regular subsets of some fixed ambient space. The theory of valuations originates from Dehn's solution to Hilbert's third problem on scissors congruence of polytopes~\cite{dehn1901} and was first formalized and systematically studied by Hadwiger~\cite{Hadwiger57} for compact convex subsets of Euclidean spaces. In a series of papers, Alesker~\cite{alesker_val_man1,alesker_val_man2,alesker_val_man4,alesker_val_man3} introduced a theory of valuations on manifolds as a generalization of the classical theory of valuations from convex geometry. 
These smooth valuations encompass classical tools from integral geometry, such as smooth measures, the Euler characteristic, and intrinsic volumes. Formally, they can be described by a pair of differential forms, an $n$-form on the manifold and an $(n-1)$-form on its cosphere bundle~\cite{bernig_quat09}. 

An important breakthrough by Alesker and Fu~\cite{alesker_val_man3} is the discovery of a product structure on the space of smooth valuations which, combined with an integration functional introduced in~\cite{alesker_val_man4}, yields a perfect bilinear pairing. This so-called \emph{Alesker-Poincaré duality} implies in particular that the space~$\val$ of smooth valuations on a smooth manifold~$X$ is densely embedded into the dual~$\gval$ of the space of compactly supported smooth valuations, so that elements of this dual space can legitimately be called \emph{generalized valuations}~\cite{alesker_val_man4}. It has been shown in~\cite{BB07} that a generalized valuation can be described as a pair of currents, an $n$-current on the manifold and an $(n-1)$-current on its cosphere bundle.

Although the most general family of sets on which a smooth valuation can be evaluated is not yet clear, Alesker proved in~\cite{alesker_val_man4} that it contains the class of subanalytic sets when the manifold~$X$ is real analytic. Subanalytic sets have a tame behaviour: they can be stratified (even satisfying Whitney's conditions), they have locally finitely many connected components and so on, see \cite{hardt75, KS90}. 

Alesker has constructed an embedding with sequentially dense image~\cite{Ale17}:
\begin{equation*}
    \emCFgval:\CF \otimes \C \hookrightarrow \gval,
\end{equation*}
where~$\CF$ is the group of \emph{constructible functions} on~$X$, that is, the group of functions~$\phi:X\to \Z$ such that the sets~$\phi^{-1}(m)$ are subanalytic for all~$m\in\Z$ and the family~$\{\phi^{-1}(m)\}_{m\in \Z}$ is locally finite. This embedding is constructed using the notion of \emph{characteristic cycle}~\cite{Fu94,KS90} of a constructible function, which is an $n$-current on the cotangent bundle of the manifold that can easily be turned into a pair of currents defining a generalized valuation. 

The groups of characteristic cycles and of constructible functions on a real analytic manifold are both isomorphic to the Grothendieck group of constructible sheaves on this manifold, as shown by Kashiwara~\cite{Ka85,KS90}. Moreover, Schapira has shown in~\cite{Schapira1991} that this group isomorphism translates fundamental operations on sheaves (exterior product, pullback, pushforward, Poincaré-Verdier duality) into rich topological operations on constructible functions. Notably, the pushforward of a constructible function by a real analytic map integrates the constructible function on the fiber of the map with respect to the Euler characteristic. These operations found remarkable applications to various areas of mathematics, from the generalization of Akbulut-King's number for real semi-algebraic sets in algebraic geometry~\cite{McCP97}, to the introduction of topological integral transforms satisfying powerful injectivity theorems~\cite{curry2022many,GLM18, KP92, Sch95} used for shape description in applied geometry~\cite{CMCMR20}.

Alesker defines Poincaré-Verdier duality for smooth and generalized valuations in~\cite{alesker_val_man4} and the exterior product, the pullback and the pushforward in a follow-up paper~\cite{alesker_intgeo}. These operations generalize well-known operations from measure theory, such as the classical pushforward of smooth measures. Moreover, these operations can be combined to induce rich algebraic structures on generalized valuations. Namely, the product of generalized valuations~\cite{AB09} can be defined as the pullback of the exterior product by the diagonal embedding, and the convolution~\cite{AB17} (say, on a Euclidean space) is defined as the pushforward of the exterior product by the addition. Fu and the first author have shown in~\cite{bernig_fu06} that these operations are dual to the celebrated kinematic formulas in integral geometry~\cite{C66}, and this result was key to establishing explicit kinematic formulas on complex spaces in~\cite{BF11, bernig_fu_solanes}.

As the structure of this introduction suggests, the operations on generalized valuations were developed with the expectation that they would also generalize the operations on constructible functions coming from sheaf theory. Alesker proves this fact for Poincaré-Verdier duality, the exterior product, and for the pullback and the pushforward under rather restrictive assumptions detailed below~\cite{alesker_intgeo}.

\subsection{Contributions} 
In this article, we show that the operations on generalized valuations restrict on constructible functions to the operations coming from sheaf theory under mild transversality assumptions. We state our results below, postponing the precise definitions to relevant sections.

Since the generalized valuation associated with a constructible function is defined using characteristic cycles, proving our result reduces to translating the operations on characteristic cycles---viewed as Borel-Moore homology classes and provided in the vocabulary of microlocal sheaf theory in~\cite{KS90}---into classical operations from geometric measure theory on the currents representing generalized valuations. To do so, we use a description of cycle operations due to Schmid and Vilonen~\cite{SV96} in terms of classical operations on Borel-Moore homology classes. The key is then that characteristic cycles can be naturally considered as Borel-Moore homology classes or as currents, and that classical operations on homology classes translate into classical operations on currents under mild transversality assumptions.

\subsubsection*{Pullback and product} Our first result concerns the pullback $f^*$ of generalized valuations by a real analytic map $f$, which is defined for submersions and for immersions satisfying some transversality assumptions~\cite{alesker_intgeo}. On constructible functions, the pullback is nothing but the precomposition. We prove our result under a transversality assumption on the constructible function which implies that the pullback is well-defined in the sense of generalized valuations.
\begin{thm}\label{thm:pullback-CFgval}
    Let~$f:X\to Y$ be a real analytic map between real analytic manifolds and let~$\psi\in\CF[][Y]$. If
    \begin{enumerate}
        \item $f$ is a submersion, or
        \item $f$ is an immersion which is transverse to the strata of a Whitney stratification of~$Y$ on which~$\psi$ is constant,
    \end{enumerate}
    then the pullback~$f^*\!\emCFgval[\psi]\in\gval[X]$ is well-defined and~$f^*\!\emCFgval[\psi] = \emCFgval[f^*\psi]$.
\end{thm}
This result is proven in~\cite[Props.~3.3.4 and~3.5.12]{alesker_intgeo} in the setting of indicator functions of compact submanifolds with corners. While submanifolds with corners may not be subanalytic, they satisfy regularity properties that arbitrary subanalytic subsets do not. Specifically, Alesker's proof of~\Cref{thm:pullback-CFgval} in the case of submanifolds with corners crucially relies on the local triviality of transverse intersections of submanifolds with corners. This property is simply wrong for subanalytic sets, as shown by the existence of analytic families of non-diffeomorphic singularities; see~\cite[Chap.~II, Ex.~2.1]{Gib76}. Our proof using Schmid and Vilonen's description of the operations on characteristic cycles circumvents this obstacle.

Combined with Alesker's result on the restriction of the exterior product~\cite[Claim.~2.1.11]{alesker_intgeo}, our result on the pullback implies a similar result for the restriction of the product of generalized valuations defined in~\cite{AB09}. This result holds for constructible functions which are transverse in the sense that there are two Whitney stratifications of the manifold whose strata are pairwise transverse and on which the constructible functions are respectively constant; see~\Cref{sec:product}.
\begin{cor}\label{cor:product-CFgval}
    If~$\phi,\psi\in\CF$ are transverse, then~$\emCFgval[\phi]\cdot\emCFgval[\psi] = \emCFgval[\phi\cdot\psi]$.
\end{cor}

This result relates the product on generalized valuations associated to subanalytic sets and the intersection of such sets. It is an analogue of the result known for submanifolds with corners~\cite[Thm.~5]{AB09} which proof rested, as for the pullback, on the local triviality of transverse intersections of such submanifolds.

\subsubsection*{Pushforward and convolution} The pushforward~$f_*$ of generalized valuations by a real analytic map $f$ is well-defined when~$f$ is an immersion or when~$f$ is a submersion which is proper on the support of~$\phi$ and whose differential satisfies a transversality assumption with respect to the wave front set of the normal cycle of~$\phi$; see~\cite{alesker_intgeo}. To prove our result on the pushforward, we first need to assume transversality of the differential of~$f$ to the strata defining the normal cycle. However, on the contrary to the pullback, this alone will not ensure that the pushforward is well-defined in the sense of generalized valuations. As a consequence, we also require the stratification of the normal cycle to be \emph{angular}, that is, to satisfy that the wave front set of the normal cycle is contained in the union of the conormal bundles to the strata; see~\Cref{def:transversality-current} and~\Cref{thm:pushforward-CFgval}. We show that angular stratifications of subanalytic currents always exist in~\Cref{sec:WF-subanalytic-currents} using desingularization.
\begin{thm}\label{main:pushforward}
    Let~$f:X\to Y$ be a real analytic map between real analytic manifolds and~$\phi\in\CF$. If
    \begin{enumerate}
        \item $f$ is an immersion, or
        \item $f$ is a submersion which is proper on the support of~$\phi$ and whose differential is transverse to the strata of an angular Whitney stratification defining the normal cycle of~$\phi$,
    \end{enumerate} 
    then~$f_*\phi\in\CF[][Y]$ and~$f_*\!\emCFgval[\phi]\in\gval[Y]$ are well-defined and~$f_*\!\emCFgval[\phi] = \emCFgval[f_*\phi]$. 
\end{thm}

To prove our result, we first provide a correction to a formula of~\cite{alesker_intgeo} for the expression of the pushforward of generalized valuations by immersions. In this case, the pushforward is an extension by zero and the proof does not present any serious difficulty. The interesting case is the case of submersions, where the pushforward of constructible functions is a non-trivial operation of integration along the fibers with respect to the Euler characteristic. In~\cite{alesker_intgeo}, Alesker proves \Cref{main:pushforward} under the restrictive assumption that the map~$f$ is a linear projection between Euclidean spaces and that the function~$\phi$ is an indicator function of a compact convex subset with smooth boundary of strictly positive Gaussian curvature. The proof is based on an approximation process which crucially relies on the restrictive assumptions on the subset. In contrast, our use of Schmid and Vilonen's description of operations on characteristic cycles allows us to bypass this obstacle once again.

Treating the pushforward yields a similar result for the convolution of generalized valuations defined in~\cite{AB17}. Recall that the theory of Lie groups (see for instance~\cite[Part~II, Chap.~IV, Sec.~5]{serre64}) ensures that if $G$ is a Lie group acting transitively on a smooth manifold~$X$, then~$G$ and~$X$ are naturally endowed with unique real analytic structures and the action of~$G$ on~$X$ is real analytic. In such a situation, one can define the convolution~$-*-$ (of constructible functions or of generalized valuations) as the pushforward of the exterior product by the action~$G \times X \to X$. Note that the action is a submersion by transitivity. Combined with Alesker's result on the exterior product, \Cref{main:pushforward} implies:
\begin{cor}\label{cor:convolution-CFgval}
	Let $G$ be a Lie group acting transitively on a smooth manifold~$X$. For~$\phi \in \mathrm{CF}(G)$ and~$\psi \in \mathrm{CF}(X)$, we have: 
	\begin{displaymath}
		[\phi] * [\psi] = [\phi * \psi],
	\end{displaymath}
    whenever the action~$a:G\times X \to X$ and $\phi \boxtimes \psi$ satisfy assumption (ii) of~\Cref{main:pushforward}.
\end{cor}
This result relates the convolution of generalized valuations with the convolution of constructible functions, which seems to be the correct replacement for the notion of Minkowski sum.

\subsubsection*{Additive and multiplicative kinematic formulas} 

Let $V$ be an $n$-dimensional Euclidean vector space and let $G$ be a closed subgroup of $\mathrm{O}(n)$ that acts transitively on the unit sphere. We endow $G$ with the Haar probability measure and $\bar G:=G \ltimes V$ with the product of Haar probability measure and Lebesgue measure. 

The space of continuous, translation- and $G$-invariant valuations on convex bodies, to be denoted by $\Val^G$, is finite-dimensional \cite{alesker00}. We let $\mu_1,\ldots,\mu_N$ be a basis of this space. Then, for any convex bodies~$K,L\subset \R^n$ there are additive kinematic formulas:
\begin{displaymath}
		\int_G \mu_i(K + gL) dg  =\sum_{k,l} c_{k,l}^i \mu_k(K)\mu_l(L),
\end{displaymath}
with constants $c_{k,l}^i$ that can be determined (at least in some important cases) by either the template method or using algebraic integral geometry. It turns out that under our assumption on $G$, elements of~$\Val^G$ are smooth valuations, and they can be evaluated at compactly supported constructible functions. As we argued above, the correct replacement for the Minkowski sum of convex bodies is the convolution product of constructible functions. Moreover, we can use our result on the convolution to prove additive kinematic formulas for constructible functions:
\begin{thm}\label{main:additive-kinematic-formulas}
	If $\phi_1,\phi_2 \in \mathrm{CF}(V)$ are compactly supported, then with the same constants $c_{k,l}^i$ as above we have: 
	\begin{displaymath}
		\int_G \mu_i(\phi_1 * g_*\phi_2) dg  =\sum_{k,l} c_{k,l}^i \mu_k(\phi_1)\mu_l(\phi_2).
	\end{displaymath} 
\end{thm}

The proof of this theorem uses a map from the space of compactly supported generalized valuations to translation-invariant generalized valuations, that was considered earlier in the special cases of smooth valuations \cite{AB17} as well as polytopes \cite{bernig_faifman16}. 
One difficulty is to prove that for almost all $g\in G$, the generalized valuations associated to $\phi_1$ and $g_*\phi_2$ are transversal, so that their convolution product exists.

We note that it may be possible to prove the theorem with a more classical approach, namely using a Hadwiger type characterization of invariant valuations. Such a Hadwiger type theorem is indeed claimed  in the context of constructible functions for $G=\mathrm{SO}(n)$ \cite[Lemma 12]{baryshnikov_ghrist_wright}, but the notion of {\it conormal continuity} may be hard to check for the kinematic integral. Moreover, our approach has the advantage to work for other transitive groups, and also on the Lie groups $S^3$ and $\mathrm{SO}(3)$. These two groups are the only compact connected Lie groups of dimension $ \geq 2$ such that the space of bi-invariant smooth valuations is finite-dimensional \cite[Thm.~B]{bernig_faifman_kotrbaty}.

We state the formulas for $S^3$, the case of $\mathrm{SO}(3)$ being similar. A natural basis of the space of $\mathrm{SO}(4)$-invariant smooth valuations consists of the Crofton valuations $\nu_i$ defined for~$i\in\{0,1,2,3\}$ by
\begin{displaymath}
	\nu_i(X)=\int_{\Geod_{3-i}(S^3)} \chi(X \cap E) dE,
\end{displaymath}
where $\Geod_{3-i}(S^3)$ is the manifold of totally geodesic submanifolds of $S^3$ of dimension $(3-i)$, endowed with the $\mathrm{SO}(4)$-invariant probability measure.

\begin{thm} \label{thm_additive_s3}
	There are constants $d_{k,l}^i$ such that for~$\phi_1,\phi_2\in\CF[][S^3]$ the following multiplicative kinematic formulas hold:
	\begin{displaymath}
		\int_{\mathrm{SO}(4)} \nu_i(\phi_1 * g_* \phi_2) dg  =\sum_{k,l} d_{k,l}^i \nu_k(\phi_1) \nu_l(\phi_2).
	\end{displaymath}
Setting $m(\nu_i)=\sum d_{k,l}^i \nu_k \otimes \nu_l$, we obtain the following explicit formulas 
	\begin{align*}
		m(\nu_0) & = \nu_0 \otimes \nu_0,\\
		m(\nu_1) & = \nu_0 \otimes \nu_1+\nu_1 \otimes \nu_0-(\nu_1 \otimes \nu_2+\nu_2 \otimes \nu_1)+2(\nu_2 \otimes \nu_3+\nu_3 \otimes \nu_2),\\
		m(\nu_2) & =\nu_0 \otimes \nu_2+\nu_2 \otimes \nu_0 +\frac{\pi^2}{8} \nu_1 \otimes \nu_1-\frac{\pi^2}{4} (\nu_1 \otimes \nu_3+\nu_3 \otimes \nu_1)\\
		& \quad -2 \nu_2 \otimes \nu_2+\frac{\pi^2}{2} \nu_3 \otimes \nu_3,\\
		m(\nu_3) & = \nu_0 \otimes \nu_3 + \nu_3 \otimes \nu_0 +\frac12 (\nu_1 \otimes \nu_2+\nu_2 \otimes \nu_1) -(\nu_2 \otimes \nu_3+\nu_3 \otimes \nu_2).
	\end{align*}
\end{thm} 

\subsection{Plan of the paper.}
\Cref{sec:preliminaries} contains our notations and some preliminaries on Borel-Moore homology, currents, constructible functions, characteristic cycles and generalized valuations. In~\Cref{sec:WF-subanalytic-currents}, we define angular stratifications of subanalytic currents and prove their existence. For the sake of completeness, we prove in~\Cref{sec:exterior-product} that the exterior product of generalized valuations restricts to the exterior product of constructible functions~\cite[Claim~2.1.11]{alesker_intgeo}. In~\Cref{sec:pullback}, we recall the definition of the pullback of generalized valuations and prove~\Cref{thm:pullback-CFgval} and~\Cref{cor:product-CFgval}. In~\Cref{sec:pushforward}, we recall the definition of the pushforward of generalized valuations and prove~\Cref{main:pushforward} and~\Cref{cor:convolution-CFgval}. In~\Cref{sec:kinematic-formulas}, we prove additive kinematic formulas for constructible functions on flat spaces and multiplicative kinematic formulas for constructible functions on~$S^3$. 

\subsection*{Acknowledgements}

The first named author wishes to thank the IHES for the kind and productive atmosphere during his research stay in 2024, where large parts of this work have been worked out. He also wants to thank Dan Abramovich, Semyon Alesker and Thomas Wannerer for useful discussions. The second named author wishes to thank Andre Belotto da Silva, Antoine Commaret, Jean-Marc Delort, Pierre Schapira and Shu Shen for useful discussions.

\setcounter{thm}{0}
\renewcommand{\thethm}{\thesection.\arabic{thm}}

\section{Preliminaries}\label{sec:preliminaries}
In this section, we introduce our notations and recall known definitions and results used throughout the text.

\subsection{Notations}
Unless explicitly stated otherwise, manifolds are smooth and without boundary.  We use the following conventions.
\begin{itemize}
    \item If~$V$ is a vector space and~$V^*$ its dual space, then for all~$\alpha\in V^*$ and~$x\in V$ we denote the natural pairing by~$\dualdot{\alpha, x} = \langle x,\alpha\rangle = \alpha(x)$.
    \item $X^n$ for a smooth manifold $X$ of dimension~$n$.
    \item If $X$ is a smooth manifold and $Y$ a submanifold of $X$, we denote by~$T^*_Y X \subset T^*X$ the conormal bundle of $Y$.
    \item If~$f:X\to Y$ is a smooth map between smooth manifolds, we consider the following classical diagram:
    \begin{equation}\label{eq:df-diag}
        \begin{tikzcd}
            T^*X & X \times_Y T^*Y \arrow[l, "\ df"'] \arrow[r, "\tau"] & T^*Y
        \end{tikzcd}.
    \end{equation}
    \item Given a real vector bundle~$E$ over~$X$, we denote by~$\zero_X$ or simply~$\zero$ its zero section. Moreover, we denote by~$\P_+(E)$ the oriented projectivization of~$E$, that is, the quotient~$(E\setminus \zero_X)/\R_{>0}$ where~$\R_{>0}$ is the group of positive real numbers acting by multiplication on the fibers, i.e., by~$(p,\xi)\mapsto (p,\lambda\xi)$ for any~$\lambda\in\R_{>0}$ and~$(p,\xi)\in E\setminus \zero_X$. We denote by~$(p,[\xi])$ the equivalence class of $(p,\xi)$. 
    \item When~$E$ is the cotangent bundle~$T^*X$ of a smooth manifold~$X$, we simply denote by~$\P_X = \P_+(T^*X)$ the cosphere bundle of~$X$ and by~$\pi_X$ the canonical map~$\P_X \to X$.
    \item The space of complex-valued differential $k$-forms on~$X$ is denoted by~$\Omega^k(X)$.  
    \item The indicator function of a subset~$S\subset X$ is denoted by~$\1_S$. 
    \item The interior of~$S$ inside~$X$ is denoted by~$\Int(S)$.
    \item When dealing with real analytic manifolds, stratifications are always assumed subanalytic.
\end{itemize}

\subsection{Orientations.} 
To keep the notation as simple as possible, we will assume that our manifolds are oriented. Without the choice of an orientation, the normal cycle, the characteristic cycle, and the differential forms on $X$ and~$\P_X$ have to be twisted by a section of the orientation bundle of $X$. With this modification, all formulas and statements remain true in the non-oriented case.

As in~\cite{alesker_intgeo,AB09}, we use the following conventions on orientations. 

\subsubsection*{Boundary.} If $X$ is an oriented manifold with boundary, then its boundary~$\partial X$ is endowed with the Stokes orientation~\cite[Sec.~3.2]{GP10}. 

\subsubsection*{Quotient.} Let~$G$ be an oriented Lie group acting smoothly, freely and properly on an oriented manifold~$X$. Assume further that the action of~$G$ on~$X$ is orientation-preserving. Then, the quotient manifold~$X/G$ is orientable. We endow~$X/G$ with the following orientation. For any~$p\in X$, we have a canonical isomorphism~$T_pX \cong T_p(G\cdot p)\oplus T_{\pi(p)}(X/G)$. We use this as a convention for~$X/G$, namely~$\mathrm{or}(T_p(G\cdot p))\wedge \mathrm{or}(T_{\pi(p)}(X/G)) = \mathrm{or}(T_p X)$. For instance, the unit sphere can be oriented as the boundary of the unit ball or as the quotient of the Euclidean space minus the origin by the action of the multiplicative group of positive real numbers. Using the above conventions, these two orientations coincide. 

\subsubsection*{Intersection.} Let~$X$ and~$Y$ be oriented submanifolds of an oriented manifold~$Z$. If the intersection~$X\cap Y$ is transverse then it is induced an orientation as the preimage of~$X$ by the embedding~$Y\hookrightarrow Z$ following the conventions of~\cite[Sec.~3.2]{GP10}. This convention on orientations coincide with the one of \cite[Ex.~11.12]{Bre93} which we recall here for the sake of clarity. Let~$z \in X\cap Y$. Choose a basis of~$T_z(X\cap Y)$. Complete this basis to a positive basis of~$T_zY$. Then, add vectors of~$T_zX$ so that the chosen basis of~$T_z(X\cap Y)$ followed by these vectors is a positive basis of~$T_zX$. By transversality, the so obtained collection of tangent vectors is a basis of~$T_zZ$. If it is positively oriented, give~$X\cap Y$ the orientation given by the basis chosen in the first place. Otherwise, give it the opposite orientation.  

\subsubsection*{Fiber product.} Let $f:X \to Z$ and $g:Y\to Z$ be two smooth maps between oriented manifolds where at most one of~$X$ and~$Y$ may have a boundary and~$Z$ is boundaryless. Assume further that~$f\times g$ and~$(f\times g)|_{\partial(X\times Y)}$ is transversal to the diagonal~$\Delta_Z\subset Z\times Z$. Then, the \emph{fiber product} defined as~$X\times_Z Y = (f\times g)^{-1}(\Delta_Z)$ is an orientable smooth manifold endowed with the preimage orientation~\cite[Sec.~3.2]{GP10}. Informally, our conventions are such that~$\mathrm{or}(Z)\wedge \mathrm{or}(X\times_Z Y) = \mathrm{or}(X)\wedge \mathrm{or}(Y)$. These conventions agree for instance with \cite[Sec.~7.2]{abbaspour2022morse} and constitutes an ``associative'' choice of orientations when considering successive fiber products. 
In particular, when~$g:Y\to Z$ is a fiber bundle with oriented base~$Z$, oriented fiber~$F$, and total space~$Y$ oriented by the local product orientations~$U\times F$ with~$U\subset Z$ open, then the orientation induced on the pullback bundle~$X\times_Z Y$ is also given by the local product orientations of~$V\times F$ with~$V\subset X$ open.

\subsubsection*{Normal bundle.} If~$Y$ is a submanifold of an oriented manifold~$X$, then the conormal bundle~$T_Y^*X$ is oriented as any open tubular neighborhood of~$Y$ in~$X$. This correctly defines an orientation of $T_Y^*X$ as the induced orientation does not depend on the choice of auxiliary Riemannian metric used to construct tubular neighborhoods. 

If $Y$ is oriented, our choice of orientation of~$T^*_YX$ and our conventions on the orientation of fiber bundles induce an orientation of the fibers of~$T_Y^*X$. These fibers are equipped with an action of~$\R_{>0}$ which preserves the orientation, yielding an orientation on the fibers of~$\P_+(T^*_YX)$ and hence of the total space~$\P_+(T^*_YX)$.

\subsubsection*{Oriented blowup.} In the previous setting, the definition of the oriented blow\-up~$\tilde{X}$ of~$X$ along~$Y$ is recalled in \cite[Sec.~1.2]{alesker_intgeo}. As in loc. cit., we endow~$\tilde{X}$ with the orientation satisfying that the canonical embedding~$X\setminus Y \hookrightarrow \tilde{X}$ is orientation-preserving. 

\subsection{Currents}\label{sec:currents} 
Let $X$ be an $n$-dimensional manifold (possibly with boundary). The space~$\Omega_\cpt^k(X)$ of smooth and compactly supported $k$-forms on~$X$ has a natural topology and its dual space is the space~$\mathcal D_k(X)$ of $k$-currents on~$X$. If~$T \in \mathcal D_k(X)$ and~$\omega \in \Omega^k_\cpt(X)$, we write~$\langle T,\omega\rangle$ or~$\int_T \omega$ for the natural pairing. The latter notation is motivated by the fact that an oriented $k$-dimensional closed submanifold~$Y \subset X$ defines a current~$\intcur{Y}$ by~$\omega \mapsto \int_Y \omega$. The boundary~$\partial T$ of a current~$T$ is defined by~$\int_{\partial T}\omega=\int_T d\omega$. Clearly~$\partial \intcur{Y}=\intcur{\partial Y}$ by Stokes' theorem. A current~$T$ is called a \emph{cycle} if~$\partial T=0$. The support of a current is defined in the obvious way. 

The pullback and pushforward of currents are dual to the operations of pushforward and pullback of differential forms respectively, whenever the latter ones are defined. In particular, we can define the pushforward of a current by a map $f:X \to Y$ which is proper on the support of the current and the pullback by any fiber bundle map~$f:X^n\to Y^m$. More precisely, if~$\omega \in \Omega_\cpt^k(X)$, the \emph{pushforward} $f_*\omega \in \Omega_\cpt^{k+m-n}(Y)$, also called \emph{fiber integration}~\cite{BT82}, is uniquely defined by:
\begin{displaymath}
    \int_Y \alpha \wedge f_*\omega =\int_X f^*\alpha \wedge \omega , \quad \alpha \in \Omega^{n-k}(Y).
\end{displaymath}
In that case, the pullback of a current~$T\in\cur{k}[Y]$ is given by~$\dualdot{f^*T,\omega} = \dualdot{T,f_*\omega}$ for any~$\omega\in\forms[X][\cpt]{k+n-m}$.

More generally, the pullback of a current~$T$ by a smooth map is defined under some conditions on a closed conical subset~$\WF(T)\subset T^*X\setminus \zero_X$ associated to~$T$ called its \emph{wave front set} \cite{Hor03,Sato71,KKS73}. Here, we do not recall the definition of wave front sets, but only a few key results used throughout the present article, referring to \cite[Chap.~8]{Hor03} and~\cite[Sec.~2.2]{bernig_faifman16} for more details. 
If~$\Gamma\subset T^*X\setminus \zero_X$ is a closed conical subset, we denote by~$\cur{k,\Gamma}[X]$ the space of~$k$-current over~$X$ such that~$\WF(T) \subset \Gamma$. Recall~\eqref{eq:df-diag} and denote similarly~$\left.df\right|_{\del X} \colon (\del X) \times_Y T^*Y\to T^*(\del X)$.
\begin{prop}[\textnormal{\cite[Thm.~8.2.4]{Hor03}}]\label{prop:pullback-cur}
    Let~$f\colon X^n \to Y^m$ be a smooth map between oriented smooth manifolds, where~$X$ is possibly with boundary and~$Y$ is without boundary. Let~$\Gamma \subset T^*Y\setminus \zero_X$ be a closed conical subset such that:
    \begin{enumerate}[leftmargin=1cm]
        \item\label{itm:ass-pullback-inside} for all~$(x,\eta)\in X\times_Y T^*Y$ with~$\tau(x,\eta)\in\Gamma$, one has~$df^*\eta\neq 0$;
        \vspace{0.2em}
        \item\label{itm:ass-pullback-bdry} for all~$(x,\eta)\in \del X\times_Y T^*Y$ with~$\tau(x,\eta)\in\Gamma$, one has~$(df|_{\del X})^*\eta\neq 0$.
    \end{enumerate}
    \vspace{0.1em}
    Then, denoting~$f^*\Gamma := df\left(\tau^{-1}\left(\Gamma\right)\right)$ and~$\Gamma' = \left(f^*\Gamma + T^*_{\del X}X\right)\setminus \zero_X$, there exists a unique sequentially continuous map
    \begin{equation*}
        f^* \colon \cur{k,\Gamma}[Y] \to \cur{k+n-m, \Gamma'}[X]
    \end{equation*}
    extending the pullback of smooth forms, called the \emph{pullback} of currents.
\end{prop}
The \emph{intersection}~$S\cap T$ of two currents~$S$ and~$T$ on~$X$ is then defined as the pullback of the exterior product of two currents on~$X$ by the diagonal embedding~$\delta:X\hookrightarrow X\times X$ whenever it is well-defined. 

We say that the diagram of oriented smooth manifolds:
\begin{equation}\label{eq:cartesian-square}
   \begin{tikzcd}
       W \arrow[d, "h"'] \arrow[r, "h'"] & Y \arrow[d, "g"] \\
       X \arrow[r, "f"]                            & Z               
   \end{tikzcd}
\end{equation}
is a \emph{Cartesian square} if~$W\cong X\times_Z Y$ as oriented manifolds. In such a situation, we have the following base change formula:
\begin{lem}[{\cite[Lem.~1.10]{Fu90}}]
    Let~$g:Y\to Z$ be a fiber bundle and consider a Cartesian square as in~\eqref{eq:cartesian-square}. If~$f$ is proper on the support of~$T\in\cur{k}[X]$, then~$g^*f_*T= h'_* h^* T$.
\end{lem}

\subsection{Borel-Moore homology}\label{sec:BM}
If~$X^n$ is a smooth manifold, we denote by~$\BM[X][k]$ the~$k$-th Borel-Moore homology $\C$-vector space, which can be defined using the dualizing complex of the manifold~\cite[Def.~3.1.16]{KS90} or using the chain complex of locally finite singular chains with complex coefficients \cite[Cor.~V.12.21]{Br97}. We denote by~$\lfc[X][k]$ the vector space of locally finite singular~$k$-chains with complex coefficients and by~$\del_k:\lfc[X][k]\to\lfc[X][k-1]$ the corresponding boundary operator. We thus have $\BM[X][k]:=\mathrm{ker} \del_k/\mathrm{im} \del_{k+1}$.

If~$f:X \to Y$ is a smooth map which is proper on the support of~$C\in\BM[X][k]$, then the \emph{pushforward}~$f_*C\in\BM[Y][k]$ is defined in the obvious way as the image by~$f$ of the singular chains defining~$C$.

The pullback of Borel-Moore homology classes is defined via Poincaré duality; see \cite[Sec.~VI.11]{Bre93}, \cite{SV96}. Suppose~$X$ oriented and let~$Z\subset X$ be a closed subset. We denote by~$H^k_Z(X) = H^k(X,X\setminus Z)$ the group of degree-$k$ cohomology classes of~$X$ with support in~$Z$; see for instance \cite[Sec.~B.2]{Ful97}. Then, there is a Poincaré duality isomorphism~\cite[Prop.~3.3.6]{KS90}:
\begin{equation*}
    \PD:H^k_Z(X) \to \BM[Z][n-k].
\end{equation*}
If~$Z$ is only assumed locally closed, then~$H^k_Z(X)$ is defined as~$H^k_Z(X) = H^k(U,U\setminus Z)$ where~$U$ is an open subset of~$X$ containing~$Z$ as a closed subset; see \cite[Def.~2.3.8]{KS90}.

Let~$Y^m$ be an oriented smooth manifold and~$Z'\subset Y$ be a closed subset. If~$f:X\to Y$ is a smooth map, we denote by~$f^*\colon H^k_{Z'}(Y) \to H^k_{f^{-1}(Z')}(X)$ the classical pullback of cohomology classes. Then, the \emph{pullback} of Borel-Moore homology classes on~$Z'$ is defined as the composition:
\begin{equation*}
    \begin{tikzcd}
        {f^*: \BM[Z'][k]} \arrow[r, "\PD^{-1}"] & H^{m-k}_{Z'}(Y) \arrow[r,"f^*"] & H^{m-k}_{f^{-1}(Z')}(X) \arrow[r, "\PD"] & {\BM[f^{-1}(Z')][k+n-m]}.
    \end{tikzcd}
\end{equation*}

The following lemma follows from the base change formula for sheaves~\cite[Prop.~3.1.9]{KS90}.
\begin{lem}
    If the following diagram of oriented smooth manifolds:
    \begin{equation*}
       \begin{tikzcd}
           W \arrow[d, "h"'] \arrow[r, "h'"] & Y \arrow[d, "g"] \\
           X \arrow[r, "f"]                            & Z               
       \end{tikzcd},
    \end{equation*}
    is a Cartesian square and if~$f$ is proper on the support of~$C\in\BM[X][k]$, then~$g^*f_*C= h'_* h^* C$.
\end{lem}

\subsection{Subanalytic currents and Borel-Moore homology.}
In the context of real analytic manifolds, we will consider subanalytic currents. We refer to \cite{hardt75} for more information on this topic and to~\cite[Chap.~8]{KS90} for more details on subanalytic sets.
 
Let~$X$ be a real analytic manifold (possibly with boundary). A~$k$-current~$T$ on~$X$ is called \emph{subanalytic} if there exists a locally finite stratification~$\{S_\alpha\}_{\alpha\in A}$ of~$X$ together with a complex multiplicity~$m_\alpha$ and a choice of orientation for each~$k$-dimensional stratum~$S_\alpha$ such that~$T = \sum_{\alpha\in A_k} m_\alpha \intcur{S_\alpha}$, where~$A_k$ indexes the strata of dimension $k$. In such a situation, the stratification~$\{S_\alpha\}_{\alpha\in A}$ is said to \emph{define}~$T$. We denote by~$\scur{k}[X]$ the space of subanalytic~$k$-currents on~$X$ and if~$S\subset X$ is a closed subanalytic subset, we denote by~$\scur{k}[S]$ the subspace of currents supported on~$S$.

The triangulation theorem of Hardt~(\cite{hardt77}, \cite[Prop.~8.2.5]{KS90}) provides an isomorphism between the chain complex of subanalytic currents and the chain complex of locally finite singular chains with complex coefficients; see~\cite{hardt75} or~\cite[Thm.~9.2.10]{KS90}.
In addition, this isomorphism is compatible with pushforwards and pullbacks under mild assumptions, as shown by the following two lemmas due to Hardt~\cite{hardt75}; see also~\cite[App.~B]{nicolaescu10}. In what follows, we will denote by~$\bar{T}$ the Borel-Moore homology class associated to a subanalytic cycle~$T$.
\begin{lem}\label{lem:cur-BM-pushforward}
    Let~$T\in\scur{k}[X]$ be a cycle and let~$f\colon X \to Y$ be a real analytic map between oriented real analytic manifolds which is proper on~$S = \supp(T)$. Then, the pushforward~$f_*T$ is a subanalytic~$k$-current supported on~$f(S)$ and~$f_*\bar{T} = \bar{f_*T}$ in~$\BM[f(S)][k]$.
\end{lem}

For the pullback, it is proven in~\cite{hardt75} that the intersection of the subanalytic current with the integration current on the embedded submanifold is well-defined under mild transversality assumptions and corresponds to the classical intersection theory on Borel-Morel homology classes. This intersection theory of subanalytic currents is defined using the theory of slicing, but it is easy to verify that the slicing is the same as the pullback by the embedding of the fiber in the sense of~\Cref{prop:pullback-cur}. Thus we have:

\begin{lem}\label{lem:cur-BM-pullback}
    Let~$f\colon X^n \hookrightarrow Y^m$ be a closed embedding of oriented real analytic manifolds and let~$T\in\scur{k}[Y]$ be a cycle with support~$S = \supp(T)$. Assume that the map~$f$ is transverse to the strata of a Whitney stratification of~$Y$ defining~$T$. Then, the pullback~$f^*T$ of the current~$T$ is well-defined and one has~$f^*\bar{T} = \bar{f^*T}$ in~$\BM[f^{-1}(S)][k+n-m]$.
\end{lem}

Note that the assumption of this last lemma is slightly stronger than what is stated in~\cite{hardt75}, but it is sufficient for our purpose.

\subsection{Constructible functions}
As we exposed in the introduction, constructible functions are defined as integer-valued functions on a real analytic manifold~$X$ which are piecewise constant on a locally finite stratification of~$X$. If~$\phi:X\to \Z$ is constructible, the theory of stratification of subanalytic sets (\cite[Thm.~8.3.20, Exo.~VIII.12]{KS90}) ensures that there exists a Whitney stratification of~$X$ by locally closed relatively compact subanalytic submanifolds~$S_i\subset X$, $i\in I$ such that~$\phi$ is constant on each~$S_i$, i.e., there are~$m_i\in\Z$ such that:
\begin{equation}\label{eq:exp-CF}
    \phi = \sum_{i\in I} m_i\1_{S_i}.
\end{equation}
It follows from basic properties of subanalytic sets that constructible functions form a commutative ring~$\CF$ for pointwise addition and multiplication (see~\cite[Sec.~8.2]{KS90}). Let us also denote by~$\CF[\cpt]$ the subring of compactly supported constructible functions, for which the sum~\eqref{eq:exp-CF} can be taken finite and involving only compact subanalytic subsets. 

In this article, we focus on operations on constructible functions, as developped in~\cite{Viro1988} in the complex setting and by~\cite{Schapira1991} in the real analytic setting. First, consider another real analytic manifold~$Y$ and let~$\phi\in\CF$ and~$\psi\in\CF[][Y]$. The \emph{exterior product}~$\phi\boxtimes\psi\in\CF[][X\times Y]$ is the constructible function defined by~$(\phi\boxtimes\psi)(x,y) = \phi(x)\psi(y)$ for any~$(x,y)\in X\times Y$. Moreover, if~$f:X\to Y$ is a real analytic map, the \emph{pullback}~$f^*\psi\in\CF$ is defined simply by~$f^*\psi=\psi\circ f$. That these two operations send constructible functions to constructible functions follows again from well-known properties of subanalytic sets. Moreover, the pointwise product of two constructible functions on~$X$ clearly coincides with the pullback of their exterior product by the diagonal embedding~$\delta:X\hookrightarrow X\times X$.

A more involved operation is the pushforward of a constructible function by a real analytic map. First, if~$S\subset X$ is a locally closed relatively compact subanalytic subset of~$X$, then we denote by~$\chi(S)$ its Euler characteristic with compact support; see~\cite[Chap.~9]{KS90}. In particular, when~$S$ is compact, then~$\chi(S)$ is the usual Euler characteristic. Then, the \emph{integral} of a compactly supported function~$\phi\in\CF[\cpt]$ which can be written as~\eqref{eq:exp-CF} is the integer:
\begin{equation*}
    \int_X\phi\dEuler = \sum_{i\in I}m_i\cdot\chi(S_i).
\end{equation*} 
Finally, if~$f:X\to Y$ is a real analytic map which is proper on the support of~$\phi\in\CF$, the \emph{pushforward}~$f_*\phi\in\CF[][Y]$ is defined for any~$y\in Y$ by:
\begin{equation*}
    f_*\phi(y) = \int_X \1_{f^{-1}(y)} \phi \dEuler.
\end{equation*}
That the function~$f_*\phi$ is constructible is proven in~\cite[Sec.~9.7]{KS90} using sheaf theory, and can also be seen as a consequence of the cell decomposition theorem in the o-minimal structure of globally subanalytic sets~\cite[Chap.~4, (2.10)]{vdD98}. Using this operation, one can naturally define the convolution of constructible functions. For that, suppose that~$G$ is a Lie group acting on~$X$ so that the action~$a:G\times X\to X$ is real analytic. Then, the \emph{convolution} of~$\phi\in\CF[][G]$ and~$\psi\in\CF[][X]$ is defined by~$\phi*\psi = a_*(\phi\boxtimes\psi)$ whenever the map~$a$ is proper on the support of~$\phi\boxtimes\psi$.

\subsection{Characteristic cycles}\label{sec:characteristic-cycles}
Let~$X^n$ be an oriented real analytic manifold. Denote by~$\CLag{X}$ the group of \emph{Lagrangian chains} of~$X$, that is, the subgroup of~$\lfc$ generated by conical locally closed Lagrangian subanalytic submanifolds. The group of \emph{Lagrangian cycles} is the subgroup~$\Lag{X}$ of~$\BM$ generated by Lagrangian chains~$C\in\CLag{X}$ with~$\del_n C=0$.

\subsubsection*{Construction.}
To any constructible function~$\phi\in\CF$, one can associate a Lagrangian cycle~$\CC(\phi)\in\Lag{X}$ called the \emph{characteristic cycle} of~$\phi$; see~\cite[Chap.~9]{KS90} and \cite{Fu94}. The morphism~$\CC:\CF \to \Lag{X}$ induces an isomorphism between constructible functions and Lagrangian cycles, with an explicit inverse; see \cite[Thm.~8.3]{Ka85}, \cite[Thm.~9.7.11]{KS90}. We briefly recall its construction here. Denote by~$\mathcal{S}$ a Whitney stratification of~$X$ such that~$\phi$ is constant on each stratum. We denote:
\begin{equation*}
    \Lambda = \bigcup_{S\in\mathcal{S}} T^*_{S}X.
\end{equation*}
Then, the subset 
\begin{equation*}
    \Lambda^\circ = \bigcup_{S\in\mathcal{S}} \left(T^*_{S}X \setminus \bigcup_{S' \supsetneq S} \closure{T^*_{S'}X}\right)
\end{equation*}
is a smooth subanalytic subset of~$T^*X$ which is open and dense in~$\Lambda$. We denote by~$\{\Lambda_\alpha\}_{\alpha\in A}$ the collection of connected components of~$\Lambda^\circ$. Note that for each connected component~$\Lambda_\alpha$, there exists~$S_\alpha\in\mathcal{S}$ such that~$\Lambda_\alpha\subset T^*_{S_\alpha}X$. The characteristic cycle of~$\phi$ is defined as:
\begin{equation}\label{eq:def-CC}
    \CC(\phi) = \sum_{\alpha \in A} m_\alpha [\Lambda_\alpha],
\end{equation}
where~$[\Lambda_\alpha]\in\lfc[\Lambda][n]$ and~$m_\alpha\in\Z$ are certain integral multiplicities that can be defined locally using stratified Morse theory~\cite{GM88}; see~\cite{SV96}.

\subsubsection*{Operations.}
The pushforward and pullback operations on characteristic cycles have been defined in~\cite{KS90}. We recall their description as exposed in~\cite{SV96}. 
The maps from the classical diagram~\eqref{eq:df-diag} allow one to describe the effect of the pushforward of constructible functions on their characteristic cycles.
\begin{lem}[\textnormal{\cite[Prop.~9.4.2]{KS90}, \cite[(2.17)]{SV96}}]
    \label{lem:pushforward-CC}
    Let~$f\colon X^n \to Y^m$ be a real analytic map between oriented real analytic manifolds and~$\phi\in\CF$ such that~$f$ is proper on~$\supp(\phi)$. Then~$\tau_*df^*\CC(\phi)\in\BM[T^*Y][m]$ is well-defined and:
    \begin{equation*}
        \CC(f_*\phi) = \tau_*df^*\CC(\phi).
    \end{equation*}
\end{lem}
We have a similar statement for the pullback operation by real analytic maps satisfying some transversality assumptions~\cite{KS90, SV96}. See also \cite[Sec.~5]{Sab85} and \cite[Thm.~3.3]{Schu17} in the complex setting. We recall the assumption of~\cite{SV96} under a simpler terminology: 
\begin{defi}\label{def:transverse-map-CF}
    A real analytic map~$f\colon X \to Y$ is called \emph{transverse} to~$\psi\in\CF[][Y]$ (called \emph{normally non-singular} in \cite{SV96}) if there exists a Whitney stratification of~$Y$ such that~$\psi$ is constant on each stratum and~$f$ is transverse to each stratum.
\end{defi}
Under such transversality conditions, Schmid and Vilonen state:
\begin{lem}[\textnormal{\cite[Prop.~9.4.3]{KS90}, \cite[(2.20)]{SV96}}]
    \label{lem:pullback-CC}
    Let~$f\colon X^n\to Y^m$ be a real analytic map between oriented real analytic manifolds and~$\psi\in\CF[][Y]$ such that~$f$ is transverse to~$\psi$. Then~$df_*\tau^*\CC(\psi)\in\BM[T^*X][n]$ is well-defined and:
    \begin{equation*}
        \CC(f^*\psi) = df_*\tau^*\CC(\psi).
    \end{equation*}
\end{lem}
As explained in~\cite{SV96}, the assumption under which \cite[Prop.~9.4.3]{KS90} is proven is more general than the transversality assumption above, but this result will be sufficient for our purpose.

\subsection{Generalized valuations}
Let $X^n$ be a smooth oriented manifold, not necessarily real analytic. The manifold $\P_X$ is a contact manifold, with the contact hyperplane at a point $(x,[\xi])$ given by~$\ker \left( d\pi_X|_{(x,[\xi])}\right)^*\!\xi$. Vectors belonging to the contact plane are called horizontal. A differential form is called vertical if it vanishes whenever we plug in only horizontal vectors. Finally a Legendrian cycle is an $(n-1)$-cycle that vanishes on vertical forms. Note that we do not assume that $T$ is rectifiable.
	
We denote by $\submfd$ the set of compact submanifolds with corners, i.e. compact subsets that are locally diffeomorphic to a quadrant in $\R^n$. To each $P \in \submfd$ we can associate its normal cycle $N(P)$, which is an integral Legendrian cycle in $\P_X$; see for instance~\cite{alesker_intgeo}.
 
A \emph{smooth valuation} on~$X$ is a map~$\mu:\submfd\to\C$ such that there exist $\oldphi\in\forms{n}$ and $\omega\in\forms[\P_X]{n-1}$ such that for all~$P\in\submfd$, we have:
\begin{equation}\label{eq:def-val}
    \mu(P) = \int_P \oldphi+\int_{N(P)} \omega .
\end{equation}
The space of smooth valuations is denoted by~$\val$. As a quotient of the space~$\forms{n} \oplus \forms[\P_X]{n-1}$, this space is naturally endowed with a Fréchet topology. The support of a smooth valuation is defined in the obvious way and the space of compactly supported valuations is denoted by~$\val[\cpt]$. It admits a locally convex topology that is finer than the induced topology, see \cite[Section 5.1]{alesker_val_man4}. Any~$\mu\in\val[\cpt]$ can be represented as in~\eqref{eq:def-val} by compactly supported forms~$(\oldphi,\omega)\in \forms[X][\cpt]{n} \oplus \forms[\P_X][\cpt]{n-1}$, see \cite[Lemma 2.3]{bernig_quat09}. Examples of smooth valuations are the Euler characteristic, the volume (if $X$ is a Riemannian manifold), or mixed volumes with smooth reference bodies with positive curvature on $\R^n$. 

An important breakthrough obtained by Alesker and Fu~\cite{alesker_val_man3}, see also~\cite{AB09}, is the introduction of a commutative and associative product structure on~$\mathcal V^\infty(X)$ such that the Euler characteristic is the unit element. Moreover, it satisfies a version of Poincar\'e duality, i.e. the map 
\begin{displaymath}
	\mathcal V^\infty(X) \times \mathcal V^\infty_\cpt(X) \to \C, (\mu_1,\mu_2) \mapsto \int_X \mu_1 \cdot \mu_2 = (\mu_1 \cdot \mu_2)(X)
\end{displaymath}
is a perfect pairing \cite[Thm.~6.1.1]{alesker_val_man4}; see also \cite{bernig_quat09} for an alternative proof. We thus get an injection 
\begin{displaymath}
	\PD: \mathcal V^\infty(X) \hookrightarrow \big(\val[\cpt]\big)^*
\end{displaymath}
with dense image. This motivates the following definition:

\begin{defi}[{\cite[Definition 7.1.1]{alesker_val_man4}}]\label{def:gval}
    A \emph{generalized valuation} is an element of the topological dual:
    \begin{equation*}
        \gval = \big(\val[\cpt]\big)^*,
    \end{equation*}
    and this space is naturally endowed with the weak topology.
\end{defi}
The \emph{support} of a generalized valuation is defined in the obvious way and the space of compactly supported generalized valuations is denoted by~$\mathcal{V}^{-\infty}_\cpt(X)$. It has a natural locally convex topology. By \cite[Prop.~7.3.10]{alesker_val_man4},  there is a vector space isomorphism~$\mathcal{V}^{-\infty}_\cpt(X) \cong \mathcal V^\infty(X)^*$.

A result of Bröcker and the first named author~\cite{BB07} implies the following description of generalized valuations. 
\begin{prop}\label{prop:gval-as-cur}
    There is a closed embedding~$\gval \hookrightarrow \cur{n} \oplus \cur{n-1}[\P_X]$ with image the space of currents~$(C,T)$ such that:
    \begin{enumerate}
        \item $T$ is a Legendrian cycle,
        \item $\pi_{X*} T = \partial C$.
    \end{enumerate}
\end{prop} 
If $(C,T)$ is the pair corresponding to $\zeta \in \mathcal V^{-\infty}(X)$ and $\mu \in \mathcal V_\cpt^{\infty}(X)$ is given by the forms $(\oldphi,\omega)$, then 
\begin{equation}
	\langle \zeta,\mu\rangle=C(\oldphi)+T(\omega). 
\end{equation}
The \emph{wave front set} of $\zeta$ is defined by~$\WF(\zeta) = (\WF(C),\WF(T))$. Given closed conical sets~$\Lambda \subset T^*X \setminus \underline 0$ and~$\Gamma \subset T^*\P_X \setminus \underline 0$, we denote by~$\mathcal V^{-\infty}_{\Lambda,\Gamma}(X)$ the set of generalized valuations~$\zeta$ with~$\WF(\zeta) \subset (\Lambda,\Gamma)$. This space is endowed with the H\"ormander topology, see \cite{brouder_dang_helein}. 

The space $\mathcal V^\infty(X)$ also injects in this space of currents. If $\mu \in \mathcal V^\infty(X)$ is given by the forms $(\oldphi,\omega)$ as in~\eqref{eq:def-val}, then by \cite{bernig_quat09} we have:
\begin{displaymath}
	C= \pi_{X*}\omega \in C^\infty(X) \subset \cur{n}, \quad T=s^*(D\omega+\pi_X^*\oldphi) \in \Omega^{n}(\P_X) \subset \cur{n-1}[\P_X].
\end{displaymath}
Here $D:\Omega^{n-1}(\P_X) \to \Omega^n(\P_X)$ is the Rumin operator (see \cite{rumin94}), and $s:\P_X \to \P_X$ is the involution $(x,[\xi]) \mapsto (x,[-\xi])$.  

The following lemma will be useful in the proof of~\Cref{main:pushforward} in~\Cref{sec:pushforward}.
\begin{lem} \label{lemma_t_equals_0}
	Assume that~$X$ is connected and let $\psi\in\gval$ be a generalized valuation with $T=0$. Then~$\psi$ is a multiple of the Euler characteristic. 
\end{lem}

\proof 
Since $\partial C=\pi_{X*}T=0$, the constancy theorem ~\cite[(4.1.7)]{Fed69} implies that $C=\lambda \intcur{X}$. As noted above, a compactly supported smooth valuation $\mu$ can be represented by a pair $(\oldphi,\omega)$ of compactly supported forms. Therefore,
\begin{displaymath}
	\langle \psi,\mu\rangle=C(\oldphi)+T(\omega)=\lambda \int_X \oldphi=\lambda \int_X \mu=\lambda \int_X \chi \cdot \mu=	\langle \PD(\lambda \chi),\mu\rangle.
\end{displaymath}  
\endproof

\subsection{Constructible functions as generalized valuations}\label{sec:CF-and-gval}
In~\cite{alesker_val_man4}, Alesker defines an embedding~$\CF \hookrightarrow\gval$ of the group of constructible functions into the space of generalized valuations which extends to an embedding~$\emCFgval:\CF \otimes \C \hookrightarrow\gval$ with sequentially dense image~\cite[Prop.~8.2.2]{alesker_val_man4}. The construction of this embedding uses the characteristic cycle construction. Following \Cref{prop:gval-as-cur}, the generalized valuation~$\emCFgval[\phi]\in\gval$ associated to~$\phi\in \CF$ is more naturally described using the \emph{normal cycle}~$N(\phi)$, which is a subanalytic Legendrian~$(n-1)$-cycle on the cosphere bundle~$\P_X$; see~\cite{Fu94}. 

We recall the relationship between normal and characteristic cycles detailed in~\cite[Sec.~4.7]{Fu94}. Following~\eqref{eq:exp-CF}, write~$\phi = \sum_{i\in I} m_i\1_{S_i}$ for integers~$m_i\in\Z$ and locally closed subanalytic subsets~$S_i$ of~$X$.
Since~$X$ is oriented, the constructible function~$\phi$ can thus naturally be seen as a current~$\basecur{\phi}$ integrating compactly supported differential~$n$-forms on the interior of the subanalytic subsets~$S_i$ that are~$n$-dimensional, and this current is clearly independent of the decomposition of~$\phi$ chosen such that~\eqref{eq:exp-CF} holds. 

Denoting~$q_X: T^*X\setminus \zero_X \to \P_X$ the quotient map and $j_X : T^*X\setminus \zero_X \hookrightarrow T^*X$ the inclusion, we can define the \emph{conification} of a subanalytic current~$T\in\cur{k}[\P_X]$ as follows. If~$S\in\scur{k}[T^*X\setminus\zero_X]$ is a subanalytic current defined by conical strata, then these strata are also subanalytic in~$T^*X$ by~\cite[Prop.~8.3.8~(i)]{KS90}, so that~$S$ naturally defines a subanalytic current on~$T^*X$, denoted by~$j_{X*}S$. We define the \emph{conification} of a subanalytic current~$T\in\scur{k}[\P_X]$ as the current~$\cone(T) = j_{X*}q_X^*T\in\scur{k+1}[T^*X]$. 

With these notations, the characteristic and normal cycles are related by the formula:
\begin{equation}\label{eq:char-to-normal-cycle}
    \CC(\phi) = \basecur{\phi} + s_*\cone(N(\phi)),
\end{equation}
where $s\colon T^*X \to T^*X$ denotes the multiplication by~$-1$ in the fibers of~$T^*X$. 
\begin{defi}\label{def:emCFgval}
    The generalized valuation~$\emCFgval[\phi]$ associated to~$\phi\in \CF$ corresponds to the pair of currents~$(C,T) = (\basecur{\phi},N(\phi))\in \cur{n} \oplus \cur{n-1}[\P_X]$.
\end{defi}
Note that the above definition of the embedding~$\emCFgval:\CF \hookrightarrow\gval$ corresponds to the sign conventions of \cite{Ale14}.

\begin{ex} \label{ex_evaluate_on_constructible}
    Let~$\phi=\1_S$ for some subanalytic subset~$S\subset X$. The generalized valuation~$\emCFgval[\phi]\in\gval[X]$ is represented by the pair of currents~$(\intcur{\Int(S)},N(S))$ where~$\Int(S)$ denotes the interior of~$S$ inside~$X$ (which is empty when~$\dim(S)<n$). If in addition~$S$ is a compact submanifold with corners, then for each $\mu \in \mathcal V_\cpt^\infty(X)$ we have~$\langle [\phi],\mu\rangle=\mu(S)$. This motivates to write $\mu(\phi)$ instead of~$\langle [\phi],\mu\rangle$ for $\phi \in \CF$. To spell this out: compactly supported smooth valuations can be evaluated on constructible functions.
\end{ex}

\section{Wave front set of subanalytic currents}\label{sec:WF-subanalytic-currents}
When considering pullbacks of subanalytic currents, we will use two notions of transversality. One is formulated in terms of geometric transversality of real analytic maps with the strata defining the current and the other formulated in terms of wave front sets. In this section, we show using desingularization that one can always stratify the support of a subanalytic current so that its wave front set is contained in the union of conormal bundles to the strata. For such a stratification, the first notion of transversality implies the second.

\begin{defi}
	Let~$T$ be a subanalytic current on~$X$. A locally finite stratification~$\{S_\alpha\}_{\alpha\in A}$ of~$X$ defining~$T$ is called \emph{angular} if we have:
	\begin{equation*}
		\WF(T) \subset \bigcup_{\alpha\in A} T^*_{S_\alpha} X \setminus \zero_{S_\alpha}.
	\end{equation*}
\end{defi}

Note that if a given stratification is angular, then so is any refinement of this stratification. In particular, one can always refine an angular stratification to make it also satisfy Whitney's conditions. Note also that angularity is preserved under diffeomorphisms.

\begin{lem} \label{lemma_product_angular_stratifications}
	If $T_1$ is a subanalytic current on $X_1$ with an angular stratification $\{S^1_\alpha\}_{\alpha \in A}$, 
	and $T_2$ is a subanalytic current on $X_2$ with an angular stratification $\{S^2_\beta\}_{\beta \in B}$, 
	then the product stratification $\{S^1_\alpha \times S^2_\beta\}_{\alpha \in A,\beta \in B}$ is an angular stratification of the subanalytic current $T_1 \boxtimes T_2$ on $X_1 \times X_2$.
\end{lem}

\proof
This is obvious from the inclusion 
\begin{displaymath}
	\WF(T_1 \boxtimes T_2) \subset (\WF(T_1) \times \WF(T_2)) \cup (\WF(T_1) \times \zero_{X_2}) \cup (\zero_{X_1} \times \WF(T_2)),
\end{displaymath}
see \cite[Thm.~8.2.9]{Hor03}.
\endproof

\begin{defi}\label{def:transversality-current}
    We say that a real analytic map~$f\colon X \to Y$ is \emph{angularly transverse} to a subanalytic current~$T\in\scur{k}[Y]$ if there exists an angular Whitney stratification of~$Y$ defining~$T$ such that~$f$ is transverse to each stratum.
\end{defi}

If~$f$ is angularly transverse to the subanalytic current~$T$, then~$f$ and~$\Gamma = \WF(T)$ satisfy the assumptions of~\Cref{prop:pullback-cur}. We now prove that angular stratifications always exist.
\begin{prop}\label{prop:WF-subanalytic}
    Any subanalytic current admits an angular stratification.
\end{prop} 
We do not know if assuming the transversality of~$f$ with the strata of \emph{some} stratification entails the existence of an angular stratification with the same property.

Our proof of~\Cref{prop:WF-subanalytic} is a combination of a functoriality property (\Cref{lem:pushforward}) and of global smoothings of subanalytic sets proven by Bierstone and Parusi\'nski (\Cref{thm:global-smoothing}).
\begin{lem}\label{lem:pushforward}
    If a subanalytic current admits an angular stratification, then so does its pushforward by a proper real analytic map.
\end{lem}
\begin{proof}
    Let~$f\colon X\to Y$ be a proper analytic map between real analytic manifolds and~$T$ be a subanalytic current on~$X$. Denote by~$\mathcal S= (S_\alpha)_{\alpha\in A}$ an angular stratification of~$X$ defining~$T$. Up to refining~$\mathcal S$ if necessary, one can consider a stratification $\mathcal R = (R_\beta)_{\beta\in B}$ of~$Y$ such that each $f^{-1}(R_\beta)$ is a union of strata of~$\mathcal S$ and each restriction~$f_{|S_\alpha} : S_\alpha \to R_\beta$ is an analytic submersion; see e.g.~\cite[Part~I, Sec.~1.7]{GM88}. 
    By~\cite[Chap.~VI, Prop.~3.9]{guillemin77} we have that:
    \begin{equation}\label{eq:pushforward-WF}
        \WF(f_*T) \subset f_*\WF(T) := \tau\left(df^{-1}\left(\WF\left(T \right) \cup \zero_X \right)\right)\setminus \zero_Y.
    \end{equation}
    It then follows directly from the facts that~$\mathcal S$ is angular and that~$f_{|S_\alpha}:S_\alpha\to R_\beta$ is a submersion that the fiber of the right-hand side above any~$y\in R_\beta$ is contained in~$(T^*_{R_\beta}Y)_y$.
\end{proof}

We only state a special case of the result of Bierstone and Parusi\'nski which is sufficient for our purpose.
\begin{thm}[\textnormal{\cite[Thm.~1.2]{bierstone18}}]
    \label{thm:global-smoothing}
    Let~$X$ be a real analytic manifold of dimension~$n$ and let~$S$ be a closed subanalytic subset of~$X$ of dimension~$k$. Then, there exist a real analytic manifold~$\tilde X$ of dimension~$n$, a proper real analytic map~$\phi : \tilde X\to X$ such that~$\phi$ induces an isomorphism from a disjoint union of relative interiors of $k$-dimensional submanifolds with corners of~$\tilde X$ onto an open subanalytic subset~$U\subset S$ such that~$\dim(S\setminus U) < k$.
\end{thm}

\begin{proof}[Proof of~\Cref{prop:WF-subanalytic}]
    Let~$T$ be a subanalytic $k$-current on~$X$. We can assume without loss of generality that~$T$ is the integration over a closed subanalytic subset~$S\subset X$ of dimension~$k$. Note that we will not care about orientations in this proof as the result proven is invariant by a change of sign. We denote by~$\tilde U$ the disjoint union of relative interiors of $k$-dimensional subanalytic submanifolds with corners of~$\tilde X$ given by~\Cref{thm:global-smoothing} and use the same notations as in this statement. The current~$\tilde T$ of integration over~$\tilde U$ being a sum of currents integrating over submanifolds with corners, it clearly admits an angular stratification. Moreover, we claim that~$T = \phi_*\tilde T$. Indeed, for any compactly supported smooth $k$-differential form~$\omega\in\forms[M][\cpt]{k}$, we have:
    \begin{equation*}
         \dualdot{\phi_*\tilde T,\omega} = \dualdot{\tilde T,\phi^*\omega}=\int_{\tilde U} \ \phi^*\omega = \int_{U} \ \omega = \int_{S} \ \omega = \dualdot{T,\omega}.
    \end{equation*}
    Hence the existence of an angular stratification for~$T$ by \Cref{lem:pushforward}.
\end{proof}

\section{Exterior product}\label{sec:exterior-product}
In this section, we prove that the exterior product of generalized valuations restricts to the exterior product of constructible functions, which was claimed in~\cite[Claim~2.1.11]{alesker_intgeo}. 

We follow the notations of \cite{alesker_intgeo, AB09}. Let~$X_1$ and~$X_2$ be two real analytic manifolds and denote~$X=X_1\times X_2$. The projections are denoted by~$\tilde p_i:X_1 \times X_2 \to X_i$ for~$i=1,2$.

Let $\mathcal{M}_1=\{(x_1,x_2,[\xi_1:0])\} \subset \P_X$ and $\mathcal{M}_2=\{(x_1,x_2,[0:\xi_2])\} \subset \P_X$. Consider the oriented blowup~$F:\hat{\P}_X \to \P_X$ along $\mathcal M:=\mathcal{M}_1 \cup \mathcal{M}_2$. The space~$\hat{\P}_X$ is the closure inside~$\P_X \times_X (\P_{X_1} \times \P_{X_2})$ of the set~$\{(x_1,x_2,[\xi_1:\xi_2],[\xi_1],[\xi_2])\}$. Then~$\hat{\P}_X$ is a manifold of dimension $2(\dim X_1+\dim X_2)-1$ with boundary, where the boundary is given by~$\mathcal N:=\mathcal N_1 \cup \mathcal N_2$ with:
\begin{align*}
	\mathcal N_1 & = \{(x_1,x_2,[\xi_1:0],[\xi_1],[\eta_2]\},\\
	\mathcal N_2 & = \{(x_1,x_2,[0:\xi_2],[\eta_1],[\xi_2]\}.
\end{align*}
The map $F:\hat{\P}_X \twoheadrightarrow \P_X$ is given by: 
\begin{displaymath}
	F(x_1,x_2,[\xi_1:\xi_2],[\eta_1],[\eta_2]):=(x_1,x_2,[\xi_1:\xi_2]).
\end{displaymath}
Moreover, we have a map $\Phi:\hat{\P}_X \twoheadrightarrow \P_{X_1} \times \P_{X_2}$ given by:
\begin{displaymath}
	\Phi(x_1,x_2,[\xi_1:\xi_2],[\eta_1],[\eta_2]):=(x_1,[\eta_1],x_2,[\eta_2]).
\end{displaymath}
Finally, we have the obvious maps:
\begin{equation*}
	\begin{tikzcd}
		\P_{X_1} \times X_2 \arrow[r, "i_1",hook]  \arrow[d, "p_1"',two heads]  & \P_X \arrow[d, "\pi_X"',two heads] & X_1 \times \P_{X_2} \arrow[l, "i_2"', hook'] \arrow[d, "p_2"',two heads] \\
		\P_{X_1} & X & \P_{X_2}
	\end{tikzcd}
\end{equation*}
The operation~$\boxtimes:\gval[X_1]\times\gval[X_2]\to \gval[X]$ is defined as follows.
\begin{defi}[{\cite[Sec.~2]{alesker_intgeo}}] \label{def_ext_prod}
    For any pair of generalized valuations~$(\zeta_1,\zeta_2) \in \gval[X_1]\times\gval[X_2]$ respectively described by pairs of currents~$(C_1,T_1)$ and~$(C_2,T_2)$, the \emph{exterior product}~$\zeta_1\boxtimes\zeta_2\in\gval$ is defined by the pair of currents:
    \begin{equation*}
        \begin{split}
            C &= C_1\boxtimes C_2,\\
            T &= F_*\Phi^*(T_1\boxtimes T_2) + (\tilde{p}_1\circ\pi_X)^*C_1 \cap i_{2*}p_2^*T_2 + i_{1*}p_1^*T_1 \cap (\tilde{p}_2\circ\pi_X)^*C_2.
        \end{split}
    \end{equation*}
\end{defi} 

We can now prove the main result of this section:

\begin{prop}[\textnormal{\cite[Claim~2.1.11]{alesker_intgeo}}]\label{prop:exterior-product-CFgval}
    Let~$\phi_1\in\CF[][X_1]$ and~$\phi_2\in\CF[][X_2]$ be constructible functions on real analytic manifolds. We have:
    \begin{equation*}
        \emCFgval[\phi_1]\boxtimes\emCFgval[\phi_2] = \emCFgval[\phi_1\boxtimes\phi_2].
    \end{equation*}
\end{prop}
\begin{proof}
    Following \Cref{def:emCFgval}, denote~$(C_i,T_i) = (C_{\phi_i},N(\phi_i))$ for~$i=1,2$. We must prove:
    \begin{align}
        C_{\phi_1\boxtimes\phi_2} &= C_1 \boxtimes C_2,\label{eq:ext-pdt-base}\\[0.5em]
        \begin{split}
            N(\phi_1\boxtimes\phi_2)  &= F_*\Phi^*(T_1\boxtimes T_2) + (\tilde{p}_1\circ\pi_X)^*C_1 \cap i_{2*}p_2^*T_2 \\
            &\hspace{1.5cm} + i_{1*}p_1^*T_1 \cap (\tilde{p}_2\circ\pi_X)^*C_2. 
        \end{split}\label{eq:ext-pdt-normal}
    \end{align}
    
    The first equality \eqref{eq:ext-pdt-base} follows readily from the definition of~$C_{\phi_1\boxtimes\phi_2}$. For~\eqref{eq:ext-pdt-normal}, recall that characteristic cycles behave well with respect to exterior product~\cite[Thm.~4.5]{Fu94}, \cite[(9.4.1)]{KS90}:
    \begin{equation*}
        \CC(\phi_1\boxtimes\phi_2) = \CC(\phi_1)\boxtimes \CC(\phi_2).
    \end{equation*}
    By the formula relating characteristic and normal cycles \eqref{eq:char-to-normal-cycle}, and denoting again by~$s:T^*X \to T^*X$ the multiplication by~$-1$ in the fibers, we have:
    \begin{equation*}
        s_*\CC(\phi_1\boxtimes\phi_2) = C_1 \boxtimes C_2 + C_1\boxtimes \cone(T_2) + \cone(T_1) \boxtimes C_2 + \cone(T_1)\boxtimes\cone(T_2),
    \end{equation*}
    and the sum of the last three terms is~$\cone(N(\phi_1\boxtimes\phi_2))$. Thus~\eqref{eq:ext-pdt-normal} will follow from:
    \begin{align}
        \cone(T_1)\boxtimes\cone(T_2) &= \cone\big(F_*\Phi^*(T_1\boxtimes T_2)\big), \label{eq:normal-normal}\\
        C_1\boxtimes \cone(T_2) &= \cone\big((\tilde{p}_1\circ\pi_X)^*C_1 \cap i_{2*}p_2^*T_2\big),\label{eq:base-normal}\\
        \cone(T_1) \boxtimes C_2 &= \cone\big(i_{1*}p_1^*T_1 \cap (\tilde{p}_2\circ\pi_X)^*C_2\big).\label{eq:normal-base}
    \end{align}
    
    To prove~\eqref{eq:normal-normal}, denote by~$\hat{\pi}:\hat{T^*X} \to \hat{\P}_X$ the pullback of the fiber bundle~$q_X:T^*X\setminus \zero_X \to \P_X$ by the smooth map~$F:\hat{\P}_X\to \P_X$. More explicitly, this pullback bundle can be described as follows. The total space~$\hat{T^*X}$ is the manifold with boundary of dimension~$2\dim(X_1)+2\dim(X_2)$ such that:
    \begin{align*}
        \Int(\hat{T^*X}) &=\big\{(x_1,x_2,\xi_1,\xi_2)\in T^*X \colon \xi_1\ne 0 \mbox{ and } \xi_2\ne 0\big\}, \\
        \partial\hat{T^*X} &= \big(T^*X_1\setminus \zero_{X_1}\big)\times \P_{X_2} \ \cup \  \P_{X_1} \times \big(T^*X_2\setminus \zero_{X_2}\big),
    \end{align*}
    and the map~$\hat{\pi}:\hat{T^*X}\to\hat{\P}_X$ is given by:
    \begin{align*}
        \hat{\pi}(x_1,x_2,\xi_1,\xi_2) &= (x_1,x_2,[\xi_1:\xi_2],[\xi_1],[\xi_2])\mbox{ on } \Int(\hat{T^*X}),\\
        \hat{\pi}(x_1,x_2,\xi_1,[\eta_2]) &= (x_1,x_2,[\xi_1:0],[\xi_1],[\eta_2])\ \mbox{ on } (T^*X_1\setminus \zero_{X_1})\times \P_{X_2},\\
        \hat{\pi}(x_1, x_2, [\eta_1],\xi_2) &= (x_1,x_2,[0:\xi_2],[\eta_1],[\xi_2])\ \mbox{ on } \P_{X_1} \times (T^*X_2\setminus \zero_{X_2}).
    \end{align*}
    Moreover, the fiber bundle~$\hat{T^*X}$ is naturally equipped with a map~$\tilde{F}:\hat{T^*X}\to T^*X\setminus \zero_X$ given by:
    \begin{align*}
        \tilde{F}(x_1, x_2, \xi_1, \xi_2) &= (x_1, x_2, \xi_1, \xi_2)\mbox{ on }\Int(\hat{T^*X}),\\
        \tilde{F}(x_1, x_2, \xi_1, [\eta_2]) &= (x_1, x_2, \xi_1, 0)\ \mbox{ on } (T^*X_1\setminus \zero_{X_1})\times \P_{X_2},\\
        \tilde{F}(x_1, x_2, [\eta_1], \xi_2) &= (x_1, x_2, 0, \xi_2)\ \mbox{ on } \P_{X_1} \times (T^*X_2\setminus \zero_{X_2}).
    \end{align*}
    These maps and bundles assemble into the following commutative diagram whose bottom square is Cartesian:
    \begin{equation}\label{eq:diag-ext-product}
        \begin{tikzcd}
            &                                                                                                                               & \Int(\hat{T^*X}) \arrow[d, "\iota", hook] \arrow[ld, "j"', hook'] \arrow[rd, "j_X\circ\iota", hook] \arrow[lldd, "p"', bend right] &      \\
            & \hat{T^*X} \arrow[ld, "\Phi\circ \hat{\pi}"', two heads] \arrow[d, "\hat{\pi}", two heads] \arrow[r, "\tilde{F}"', two heads] & T^*X\setminus 0_X \arrow[d, "q_X", two heads] \arrow[r, "j_X"', hook]                                                              & T^*X \\
            \P_{X_1}\times \P_{X_2} & \hat{\P_X} \arrow[l, "\Phi", two heads] \arrow[r, "F"', two heads]                                                             & \P_X                                                                                                                              &     
        \end{tikzcd}
    \end{equation}
    where~$j$, $j_X$ and~$\iota$ denote canonical inclusions and~$p:=\Phi\circ\hat{\pi}\circ j$. 
    Note also that we have canonical identifications~$\Int(\hat{T^*X}) =(T^*X_1\setminus \zero_{X_1}) \times (T^*X_2\setminus \zero_{X_2})$,
    \begin{align*}
        p=q_{X_1}\times q_{X_2} &: (T^*X_1\setminus \zero_{X_1}) \times (T^*X_2\setminus \zero_{X_2})\to \P_{X_1}\times\P_{X_2},\\
        j_X\circ\iota = j_{X_1}\times j_{X_2} &: (T^*X_1\setminus \zero_{X_1}) \times (T^*X_2\setminus \zero_{X_2})\to T^*X,
    \end{align*}
    which entail:
    \begin{equation*}
        \cone(T_1)\boxtimes\cone(T_2) = j_{X*}\iota_*p^*(T_1\boxtimes T_2) = j_{X*}\tilde{F}_* j_* p^*(T_1\boxtimes T_2).
    \end{equation*}
    
    Moreover, one has the equality of maps~$j_*p^* = (\Phi\circ\hat{\pi})^*$ from the space of currents on~$\P_{X_1}\times \P_{X_2}$ to the space of currents on~$\hat{T^*X}$. Indeed, for any compactly supported smooth differential form~$\omega$ on~$\hat{T^*X}$ we have that~$(\Phi\circ\hat{\pi})_*\omega = p_*j^*\omega$. This last equality follows readily from the fact that the fibers of~$\Phi\circ\hat{\pi}$ are the closure of the fibers of~$p$ in~$\hat{T^*X}$ so that the integration along the fibers of~$\Phi\circ\hat{\pi}$ of the smooth form~$\omega$ and the integration along the fibers of~$p$ of the restriction~$j^*\omega$ of~$\omega$ to the interior of~$\hat{T^*X}$ coincide. 

    Since the bottom square of~\eqref{eq:diag-ext-product} is Cartesian, one has that
    \begin{equation*}
        q_X^* F_*\Phi^*(T_1\boxtimes T_2) = \tilde{F}_*(\Phi\circ\hat{\pi})^*(T_1\boxtimes T_2),
    \end{equation*} 
    so that we have proven:
    \begin{align*}
        \cone(T_1)\boxtimes\cone(T_2) 
        &= j_{X*}\tilde{F}_* (\Phi\circ\hat{\pi})^*(T_1\boxtimes T_2)\\
        &=j_{X*}q_X^* F_*\Phi^*(T_1\boxtimes T_2) \\
        &=\cone(F_*\Phi^*(T_1\boxtimes T_2)).
    \end{align*}
    
    To prove~\eqref{eq:base-normal} (and~\eqref{eq:normal-base} by symmetry), consider the commutative diagram:
    \begin{equation*}
        \begin{tikzcd}
            & T^*X_1 \times T^*X_2                               & X_1 \times T^*X_2 \arrow[r, "p''_2", two heads] \arrow[l, "i''_2"', hook']                                                                                                               & T^*X_2                                                                                    \\
            & T^*X_1 \times (T^*X_2\setminus 0_{X_2}) \arrow[ld, "p_1\circ \pi_X\circ q_X"', two heads]\arrow[d, "q_X", two heads] \arrow[u, "j_X"', hook] & X_1 \times (T^*X_2\setminus 0_{X_2}) \arrow[d, "\id_{X_1}\times q_{X_2}", two heads] \arrow[r, "p'_2", two heads] \arrow[l, "i'_2"', hook'] \arrow[u, "\id_{X_1}\times j_{X_2}"', hook] & T^*X_2\setminus 0_{X_2} \arrow[d, "q_{X_2}", two heads] \arrow[u, "j_{X_2}"', hook] \\
        X_1 & \P_X \arrow[l, "p_1\circ \pi_X", two heads]                                                 & X_1\times \P_{X_2} \arrow[l, "i_2", hook'] \arrow[r, "p_2"', two heads]                                                                                                         & \P_{X_2}                                                                                 
        \end{tikzcd}
    \end{equation*}
    where~$j_{X_2}$, $i'_2$ and $i''_2$ are the canonical injections, and~$p_2'$, $p''_2$ and~$q_{X_2}$ the canonical submersions. The bottom left and the top right commutative squares of the above diagram are Cartesian squares. Therefore, we have:
    \begin{equation*}
        j_{X*}q_X^*i_{2*}p_2^*T_2 = j_{X*}i'_{2*}(\id_{X_1}\times q_{X_2})^*p_2^*T_2
        =i''_{2*}(p_2'')^*\cone(T_2).
    \end{equation*}
    Moreover, it follows from the definitions of the operations that
    \begin{equation*}
        i''_{2*}(p_2'')^*\cone(T_2) = \intcur{\zero_{X_1}}\boxtimes \cone(T_2),
    \end{equation*}
    and, denoting $\tilde{\pi}_{X_1} : T^*X_1 \to X_1$ the canonical submersion, we have:
    \begin{equation*}
        j_{X*}q_X^*(\tilde{p}_1\circ\pi_X)^*C_1 = \left(\tilde{\pi}_{X_1}\right)^*\!C_1\boxtimes \intcur{T^*X_2}.
    \end{equation*}
    Therefore, we have:
    \begin{align*}
        j_{X*}q_X^*\big((\tilde{p}_1\circ\pi_X)^*C_1 \cap i_{2*}p_2^*T_2\big) &= j_{X*}q_X^*\big((\tilde{p}_1\circ\pi_X)^*\!C_1 \big)\ \cap\  j_{X*}q_X^*\big(i_{2*}p_2^*T_2\big) \\
        &= \left(\tilde{\pi}_{X_1}\right)^*\!C_1\boxtimes \intcur{T^*X_2} \ \cap\  \intcur{\zero_{X_1}}\boxtimes \cone(T_2)\\
        &= C_1 \boxtimes \cone(T_2).
    \end{align*}
\end{proof}

\section{Pullback and product}\label{sec:pullback}
In this section, we recall the definition of the pullback operation on generalized valuations as defined in~\cite{alesker_intgeo} and prove that it restricts on constructible functions to the usual pullback given by precomposition under mild transversality assumptions. As a corollary, we obtain a similar result for the product of constructible functions.

\subsection{Definition}
We follow closely the exposition and the notations of~\cite{alesker_intgeo}.  Let~$X$ and~$Y$ be two oriented real analytic manifolds, let~$f:X\to Y$ be a real analytic map and let~$\zeta\in\gval[Y]$ be represented by the pair of currents~$(C,T)$. 

\subsubsection*{Submersions.} In \cite{alesker_intgeo}, the pullback of generalized valuations by submersions is defined as the adjoint operation to the pushforward of compactly supported smooth valuations. The explicit expression of the pullback as operations on the pair of currents representing the generalized valuation is proven in~\cite[Prop.~3.3.3]{alesker_intgeo}. To be more consistent with the definition of the other operations on generalized valuations, we will take this explicit expression as a definition of the pullback. Consider then the diagram of canonical maps:
\begin{equation}\label{eq:bardf-diag}
	\P_X \overset{\bar{df}}{\longleftarrow} X \times_Y \P_Y \overset{\bar{\tau}}{\longrightarrow} \P_Y.
\end{equation}
\begin{defi}[{\cite[Prop.~3.3.3]{alesker_intgeo}}]
    Suppose that~$f$ is a submersion. The \emph{pullback} of~$\zeta$ by~$f$ is the generalized valuation~$f^*\zeta\in\gval[X]$ represented by the pair of currents:
    \begin{equation*}
        (C',T') = \big(f^*C,\bar{df}_*\bar{\tau}^*T\big).
    \end{equation*}
\end{defi}

\subsubsection*{Immersions.} The pullback of generalized valuations by immersions is defined under the following transversality assumptions:
\begin{defi}[\textnormal{\cite[Def.~3.5.2]{alesker_intgeo}}]\label{def:transverse-gval}
    Denote~$\WF(\zeta) = (\Lambda, \Gamma)$.
    \begin{enumerate}
        \item A closed embedding~$f:X\to Y$ is \emph{transverse} to~$\zeta$ if:
        \begin{align*}
            \Lambda \cap T^*_X Y &= \emptyset,\\
            \Gamma \cap T^*_{X\times_Y \P_Y}\P_Y &= \emptyset, \\
            \Gamma \cap T^*_{\P_+(T^*_X Y)}\P_Y &= \emptyset.
        \end{align*}
        \item An immersion $f:X\to Y$ is \emph{transverse} to~$\zeta$ if for every~$x\in X$, there exists an open neighborhood~$U$ of~$x$ in~$X$ and an open neighborhood~$V$ of~$f(x)$ in~$Y$ such that~$f|_{U}:U\to V$ is a closed embedding which is transverse to~$\zeta|_{V}\in\gval[V]$. 
    \end{enumerate}
\end{defi}
Consider the oriented blowup~$\tilde{X\times_Y \P_Y}$ of~$X\times_Y \P_Y$ along the submanifold~$\P_+(T^*_XY)$. The total space of this blowup is a manifold with boundary of dimension~$\dim(X)+\dim(Y)-1$ and such that:
    \begin{equation}\label{eq:oriented-blowup-sphere}
        \begin{split}
            \Int\left(\tilde{X\times_Y \P_Y}\right) &= \Big\{(x, [\xi],[\xi'])\in (X\times_Y \P_Y)\times_X \P_X \colon\\
            &\hspace{4cm} df^*\xi \ne 0 \  ; \  [df^*\xi] = [\xi']\Big\}, \\
            \partial\left(\tilde{X\times_Y \P_Y}\right) &= \P_+\left(T^*_XY\right) \times_X \P_X.
        \end{split}
    \end{equation}
We denote by~$\alpha\colon \widetilde{X \times_Y \P_Y} \to X \times_Y \P_Y \hookrightarrow \P_Y$ the composition of the blowup map and of the canonical inclusion and by~$\beta \colon \widetilde{X \times_Y \P_Y} \to \P_X$ the map induced by $\overline{df}$. These maps yield a diagram:
\begin{equation} \label{eq_maps_alpha_beta}
    \P_X \overset{\beta}{\longleftarrow} \widetilde{X \times_Y \P_Y} \overset{\alpha}{\longrightarrow} \P_Y.
\end{equation}
\begin{defi}[{\cite[Claim~3.5.4]{alesker_intgeo}}]\label{def:pullback-gval-immersion}
    Suppose that~$f$ is an immersion which is transverse to~$\zeta$. The \emph{pullback} of~$\zeta$ by~$f$ is the generalized valuation~$f^*\zeta\in\gval[X]$ represented by the pair of currents:
    \begin{equation*}
        (C',T') = (f^*C,\beta_*\alpha^*T).
    \end{equation*}
\end{defi}

\subsection{Restriction to constructible functions}
In the case of a generalized valuation associated to a constructible function, the currents~$(C,T)$ are both subanalytic. In this case, the operations in~\Cref{def:pullback-gval-immersion} are thus well-defined under the assumptions that~$f$ and~$\alpha$ are transverse to the strata of a Whitney stratification defining~$\basecur{\psi}$ and~$N(\psi)$ respectively, and it is not necessary to require that the conditions of~\Cref{def:transverse-gval} are met. Moreover, in view of the construction of normal and characteristic cycles (\Cref{sec:characteristic-cycles}), it is clear that these transversality assumptions are satisfied when~$f$ is transverse to~$\psi$ in the sense of~\Cref{def:transverse-map-CF}.

\begin{thm}[see {\Cref{thm:pullback-CFgval}}]
    Let~$f:X\to Y$ be a real analytic map between real analytic manifolds and let~$\psi\in\CF[][Y]$. If any one of the following two conditions is satisfied:
    \begin{enumerate}
        \item\label{itm:thm-pullback-submersion}~$f$ is a submersion, or
        \item\label{itm:thm-pullback-immersion}~$f$ is an immersion which is transverse to~$\psi$,
    \end{enumerate}
    then the pullback~$f^*\!\emCFgval[\psi]\in\gval[X]$ is well-defined and~$f^*\!\emCFgval[\psi] = \emCFgval[f^*\psi]$.
\end{thm}

\begin{proof}
    (i) Suppose first that~$f$ is a submersion. In that case, the pullback is always well-defined. By \cite[Prop.~3.3.3]{alesker_intgeo}, we must prove:
    \begin{align}
        C_{f^*\psi} &= f^*C_\psi, \label{eq:pullback-base}\\
        N(f^*\psi) &= \bar{df}_*\bar{\tau}^*N(\psi)\label{eq:pullback-submersion-normal}.
    \end{align}
    One can reduce to local computations on an open neighborhood~$U$ of a point~$x\in X$. Using the (real analytic) normal form of submersions, one can reduce to proving the result for~$X = \R^{m+n}$, ~$Y=\R^n$ and~$f:\R^{m+n} \to \R^n$ the projection on the last coordinates. In that case, the pullback has the simple form~$f^*\psi = \1_{\R^m}\boxtimes \psi$. 
    
    In these local coordinates, equality~\eqref{eq:pullback-base} becomes obvious.
    For~\eqref{eq:pullback-submersion-normal}, note that~$N(\1_{\R^m}) = 0$, so that the result on the exterior product (\Cref{prop:exterior-product-CFgval}) ensures that:
    \begin{equation*}
        N(f^*\psi) = (\tilde{p}_1\circ\pi_X)^*\intcur{\R^m} \cap i_{2*}p_2^*N(\psi) = \intcur{\P_{\R^{m+n}}} \cap i_{2*}p_2^*N(\psi) = i_{2*}p_2^*N(\psi).
    \end{equation*}
    The fiber product~$X \times_Y \P_Y$ is isomorphic to~$\R^m \times \P_{\R^n}$ and under this identification the diagram
    \begin{equation*}
        \P_X \overset{\bar{df}}{\longleftarrow} X \times_Y \P_Y \overset{\bar{\tau}}{\longrightarrow} \P_Y
    \end{equation*}
    involved in the definition of the pullback coincides with the diagram
    \begin{equation*}
        \P_{\R^{m+n}} \overset{i_2}{\longleftarrow} \R^{m} \times \P_{\R^n} \overset{p_2}{\longrightarrow} \P_{\R^n}
    \end{equation*}
    involved in the formula for the exterior product, hence the result.

    (ii) Suppose now that~$f$ is an immersion which is transverse to~$\psi$. We must prove that~$f^*C_\psi = C_{f^*\psi}$ and~$N(f^*\psi) = \beta_*\alpha^*N(\psi)$. Since everything is local, we can assume without loss of generality that~$f$ is a closed embedding which is transverse to~$\psi$.
    
    To do so, we use Schmid and Vilonen's description of the pullback operation on characteristic cycles which ensures that~$\CC(f^*\psi) = df_*\tau^*\CC(\psi)$ as Borel-Moore homology classes; see \Cref{lem:pullback-CC}. Again, the construction of characteristic cycles and our assumption on the transversality of~$f$ to~$\psi$ imply that~$\tau$ is transverse to the strata of a subanalytic Whitney stratifcation of~$T^*Y$ defining~$\CC(\psi)$.
    Therefore, \Cref{lem:cur-BM-pullback,lem:cur-BM-pushforward} ensure that~$\CC(f^*\psi) = df_*\tau^*\CC(\psi)$ as subanalytic currents. From this last equality and the relationship between characteristic and normal cycles we deduce that~$\basecur{f^*\psi} = df_*\tau^*\basecur{\psi}$ and~$\cone(N(f^*\psi)) =  df_*\tau^*\cone(N(\psi))$. The result will then follow from:
    \begin{align*}
        df_*\tau^*\basecur{\psi} &= f^*\basecur{\psi}, \\
        df_*\tau^*\cone(N(\psi)) &=  \cone\left(\beta_*\alpha^*N(\psi)\right).
    \end{align*}

    The first equality is obvious. 
    We will prove the second in two steps. For that, let us denote by~$\tilde{q}\colon (X\times_Y T^*Y)\setminus T^*_XY \to \tilde{X\times_Y \P_Y}$ the map~$(x,\xi) \mapsto (x,[\xi],[df(\xi)])$ and by~$\tilde{j}\colon (X\times_Y T^*Y)\setminus T^*_XY\hookrightarrow X\times_Y T^*Y$ the canonical inclusion. Denote also the canonical maps~$q_X:T^*X\setminus \zero \rightarrowdbl \P_X$ and~$j_X \colon T^*X\setminus\zero \hookrightarrow T^*X$, and similarly for~$q_Y$ and~$j_Y$. These maps fit into the following commutative diagram:
    \begin{equation}\label{eq:alpha-diag-pullback}
        \begin{tikzcd}
            \P_Y                                                                 &                                                                                                       & \tilde{X \times_Y \P_Y} \arrow[ll, "\alpha"']                                                                                                                          \\
            T^*Y\setminus \zero \arrow[d, "j_Y"', hook'] \arrow[u, "q_Y", two heads] & \left(X \times_Y T^*Y\right)\setminus \zero \arrow[d, "\tilde{j}'", hook'] \arrow[l, "\tau'", hook'] & \left(X \times_Y T^*Y\right)\setminus T^*_X Y \arrow[ld, "\tilde{j}", hook'] \arrow[u, "\tilde{q}"]  \arrow[l, "\tau''", hook'] \\
            T^*Y                                                                 & X \times_Y T^*Y \arrow[l, "\tau", hook']                                                              &                                                                                                                                                                       
        \end{tikzcd}
    \end{equation}
    where $\tau',\tau'',\tilde{j}'$ are the canonical inclusions. The first step will be to prove that, denoting~$T = N(\psi)$, we have~$\tau^*\cone(T) = \tilde{j}_*\tilde{q}^*\alpha^*T$.
    The commutativity of~\eqref{eq:alpha-diag-pullback} implies that:
    \begin{equation*}
        \tilde{j}_*\tilde{q}^*\alpha^*T = \tilde{j}'_*\tau''_*(\tau'\circ\tau'')^*q_Y^*T.
    \end{equation*}
    By our transversality assumption, we have~$\supp(\tau'^*q_Y^*T)\subset (X\times_Y T^*Y)\setminus T^*_XY$, so that:
    \begin{equation*}
        \tau''_*(\tau'\circ\tau'')^*q_Y^*T = \tau'^*q_Y^*T.
    \end{equation*}
    Therefore, using that the bottom left square of~\eqref{eq:alpha-diag-pullback} is Cartesian, we get:
    \begin{equation*}
        \tilde{j}_*\tilde{q}^*\alpha^*T = \tilde{j}'_*\tau'^*q_Y^*T = \tau^*j_{Y*}q_Y^*T=\tau^* \cone(T).
    \end{equation*}

    The second step will to prove that~$df_*\tilde{j}_*\tilde{q}^*\alpha^*T = \cone(\beta_*\alpha^*T)$. To do this, consider the Cartesian square:
    \begin{equation*}
    \begin{tikzcd}
        \doublewidetilde{X \times_Y \P_Y} \arrow[r, "\pi'"] \arrow[d, "\pi''"] & \tilde{X \times_Y \P_Y} \arrow[d, "\beta"] \\
        T^*X\setminus\zero \arrow[r, "q_X"]                                & \P_X                                      
    \end{tikzcd}
    \end{equation*}
    where the total space~$\doublewidetilde{X \times_Y \P_Y}$ is a smooth manifold with boundary of dimension~$\dim(X) + \dim(Y)$ given by:
    \begin{align*}
            \Int\left(\doublewidetilde{X\times_Y \P_Y}\right) &= \Big\{(x, [\xi],\xi')\in (X\times_Y \P_Y)\times \left(T^*X\setminus\zero\right) \colon \\
            &\hspace{3.5cm} df(\xi)\ne 0 \  ; \  [df(\xi)] = [\xi']\Big\}, \\
            \partial\left(\doublewidetilde{X\times_Y \P_Y}\right) &= \P_+\left(T^*_XY\right) \times_X \left(T^*X\setminus\zero\right),
    \end{align*}
    the map~$\pi''$ is the canonical projection and the map~$\pi'$ is~$(x, [\xi],\xi')\mapsto (x, [\xi], [\xi'])$. Note that the map~$\iota:(X\times_Y T^*Y)\setminus T^*_XY \hookrightarrow \doublewidetilde{X \times_Y\P_Y}$ given by~$(x,\xi)\mapsto (x,[\xi],df(\xi))$ is an open embedding onto the interior of~$\doublewidetilde{X\times_Y\P_Y}$. Therefore, we have a commutative diagram:
    \begin{equation}\label{eq:beta-diag-pullback}
        \begin{tikzcd}
                                                                                                            & \tilde{X\times \P_Y} \arrow[r, "\beta"]                  & \P_X                                                         \\
        (X\times_Y T^*Y)\setminus T^*_XY \arrow[ru, "\tilde{q}"] \arrow[d, "\tilde{j}"', hook] \arrow[r, "\iota", hook] & \doublewidetilde{X \times_Y\P_Y} \arrow[r, "\pi''"] \arrow[u, "\pi'"] & T^*X \setminus\zero \arrow[d, "j_X", hook] \arrow[u, "q_X"'] \\
        X\times_Y T^*Y \arrow[rr, "df"]                                                                       &                                                          & T^*X                                                        
        \end{tikzcd}
    \end{equation}

    Denoting~$T' = \alpha^* T$, we must prove that~$df_*\tilde{j}_*\tilde{q}^*T' = j_{X*}q_X^*\beta_*T'$. Using the commutativity of~\eqref{eq:beta-diag-pullback}, we have:
    \begin{equation*}
        df_*\tilde{j}_*\tilde{q}^*T' = j_{X*}\pi''_*\iota_*\iota^*\pi'^*T'.
    \end{equation*}
    Again, our transversality assumptions entails~$\supp(\pi'^*T') \subset \Int(\doublewidetilde{X \times_Y \P_Y})$, so that~$\iota_*\iota^*\pi'^*T' = \pi'^*T'$. Using that the top right square of~\eqref{eq:beta-diag-pullback} is Cartesian, we finally get:
    \begin{equation*}
        df_*\tilde{j}_*\tilde{q}^*T' = j_{X*}\pi''_*\pi'^*T' = j_{X*}q_X^*\beta_*T',
    \end{equation*}
    which concludes the proof.
\end{proof}

\subsection{Product}\label{sec:product} 
In this section, we consider the product of generalized valuations defined in~\cite{AB09}.
\begin{defi}\label{def:product-gval}
    Let~$\zeta_1, \zeta_2\in\gval[X]$ be two generalized valuations such that the diagonal embedding~$\delta:X\hookrightarrow X\times X$ is transverse to $\zeta_1\boxtimes\zeta_2\in\gval[X\times X]$ in the sense of~\Cref{def:transverse-gval}. Then, the \emph{product} of $\zeta_1$ and $\zeta_2$ is the generalized valuation~$\zeta_1\cdot\zeta_2\in\gval[X]$ defined by:
    \begin{equation*}
        \zeta_1\cdot\zeta_2 = \delta^*(\zeta_1\boxtimes\zeta_2).
    \end{equation*}
\end{defi}
The definition of the product of generalized valuations is defined in~\cite{AB09} in terms of operations on the currents~$(C,T)$. That the two definitions coincide follows from the facts that (i) the transversality assumption of \Cref{def:transverse-gval} in the case of the diagonal embedding coincides with the transversality assumption for the product in~\cite[Thm.~8.3]{AB09}, (ii) the product of generalized valuations defined in~\cite[Thm.~8.3]{AB09} restricts to the product of smooth valuations, (iii) the product of smooth valuations is defined as a pullback of the exterior product by the diagonal embedding, (iv) smooth valuations are dense in the space of generalized valuations and (v) operations on generalized valuations are sequentially continuous~(\cite[Sec.~2]{alesker_intgeo},\cite[Prop.~3.5.5]{alesker_intgeo}).

\begin{proof}[Proof of~\Cref{cor:product-CFgval}.]
    Alesker's result on the restriction of the exterior product of generalized valuations to constructible functions (\Cref{prop:exterior-product-CFgval}) and our result on the restriction of the pullback (\Cref{thm:pullback-CFgval}) imply the result for the restriction of the product to constructible functions.
\end{proof}

\section{Pushforward and convolution}\label{sec:pushforward}
In this section, we recall the definition of the pushforward of generalized valuations introduced in~\cite{alesker_intgeo} with a correction to the formula for immersions. Then, we prove that the pushforward of generalized valuations restricts on constructible functions to the pushforward coming from sheaf theory under mild transversality assumptions.

\subsection{Definition}
We follow closely the exposition and the notations of~\cite{alesker_intgeo}. The pushforward of generalized valuations is defined as the adjoint of the pullback of smooth valuations using Alesker-Poincaré duality. More precisely, consider a real analytic map~$f:X^n\to Y^m$ between oriented real analytic manifolds and~$\zeta\in\gval[X]$ represented by the pair of currents~$(C,T)$. The pushforward~$f_*\zeta\in\gval[Y]$ is defined via the formula:
\begin{equation*}
    \int_Y f_*\zeta \cdot \psi = \int_X \zeta\cdot f^*\psi, \quad  \mbox{ for all }  \psi\in\val[\cpt][Y],
\end{equation*}
whenever it make sense.

\subsubsection*{Immersions} Suppose that~$f$ is an immersion and that~$n<m$, for otherwise the description is trivial. In that case, the pullback of a smooth valuation~$\psi\in\val[][Y]$ by~$f$ is again smooth, and it given for any~$P\in\submfd[X]$ by~$f_*\psi(P) = \psi(f(P))$. 

Moreover, recall the maps:
\begin{equation*}
    \P_X \overset{\beta}{\longleftarrow} \widetilde{X \times_Y \P_Y} \overset{\alpha}{\longrightarrow} \P_Y
\end{equation*}
from~\eqref{eq_maps_alpha_beta} and consider the natural maps:
\begin{displaymath}
    X \stackrel{\beta'}{\longleftarrow}  \P_+(T_X^*Y) \stackrel{\alpha'}{\longrightarrow} \P_Y.
\end{displaymath}
If~$\psi\in\val[c][Y]$ is given by the pair of differential forms~$(\oldphi,\omega)\in\forms[Y]{n} \oplus \forms[\P_Y]{n-1}$ as in~\eqref{eq:def-val}, then \cite[Prop.~3.1.2]{alesker_intgeo} ensures that the pullback~$f^*\psi$ is described by the pair~$(\beta'_*\alpha'^*\omega,\beta_*\alpha^*\omega)$. It follows then from the definition of the pushforward that:
\begin{equation*}
    \dualdot{f_*\zeta,\psi} = \dualdot{\zeta,f^*\psi} 
    = \dualdot{C,\beta'_*\alpha'^*\omega} + \dualdot{T,\beta_*\alpha^*\omega} \\
    = \dualdot{\alpha'_*\beta'^*C+\alpha_*\beta^*T,\omega}.
\end{equation*}

Therefore, we can correct a formula from \cite[Prop.~3.4.2]{alesker_intgeo}:
\begin{prop}\label{def:pushforward-immersions}
    If~$f : X^n \to Y^m$ is an immersion with $n<m$, then the generalized valuation~$f_*\zeta\in\gval[Y]$ is represented by the pair of currents: 
    \begin{equation}\label{eq:def-pushforward-immersion}
        (C',T') = \big(0,\alpha_*\beta^*T + \alpha'_* \beta'^*C\big).
    \end{equation}
\end{prop}

\subsubsection*{Submersions.}
When~$f$ is a submersion which is proper on the support of~$\zeta$, two remarks are in order. First, since the pullback of a smooth valuation by a submersion may not be smooth, the proof that such a definition makes sense requires to check that the product of generalized valuations~$\zeta\cdot f^*\psi$ is well-defined. Denoting~$\WF(\zeta) = (\Lambda,\Gamma)$, this will be the case when~$\Gamma \cap T^*_{X\times_Y \P_Y}\P_X = \emptyset$; see~\cite[Sec.~3.6]{alesker_intgeo}. Second, the pushforward of generalized valuations is defined in loc. cit. for proper submersions or for compactly supported generalized valuations. Yet, the definition clearly make sense when the map~$f$ is proper on the support of~$\zeta$. 

Moreover, denoting by~$(C',T')$ the pair of currents representing~$f_*\zeta$, it is proven in~\cite[Prop.~4.6]{AB17} that:
\begin{equation}\label{eq:T'-pushforward}
    T' = \bar{\tau}_*\bar{df}^*T.
\end{equation}
More precisely, loc. cit. proves this formula when~$f$ is proper, but the assumption that~$f$ is proper on the support of~$\zeta$ is sufficient once again. Finally, we point out that the description of~$C'$ is left as an open question in loc cit..

\subsection{Restriction to constructible functions}
In contrast to the pullback situation, the definition of pushforward of generalized valuations by submersions require a condition on wave front sets even when the generalized valuation come from a constructible function. This is due to the fact that the currents associated to the generalized valuation~$f^*\psi$ for~$\psi\in\val[\cpt][Y]$ are not subanalytic. As a consequence, we require a transversality assumption with respect to an \emph{angular} stratification of the normal cycle.
\begin{thm}[see \Cref{main:pushforward}]\label{thm:pushforward-CFgval}
    Let~$\phi\in\CF$ and let~$f:X\to Y$ be a real analytic map between real analytic manifolds. 
    If any one of the following two conditions is satisfied:
    \begin{enumerate}
        \item $f$ is an immersion, or
        \item $f$ is a submersion which is proper on the support of~$\phi$ and such that~$\bar{df}:X\times_Y\P_Y\hookrightarrow\P_X$ is angularly transverse to~$N(\phi)$,
    \end{enumerate}
    then~$f_*\phi\in\CF[][Y]$ and~$f_*\!\emCFgval[\phi]\in\gval[Y]$ are well-defined and~$f_*\!\emCFgval[\phi] = \emCFgval[f_*\phi]$. 
\end{thm}
\begin{proof}
    Recall that~$(C',T')$ are associated to~$f_*\!\emCFgval[\phi]$ and denote~$(C,T) = (C_{\phi}, N(\phi))$ and~$(C'',T'') = (C_{f_*\phi}, N(f_*\phi))$.
    
    (i) Assume first that~$f$ is an immersion. By additivity and properness of~$f$ on the support of~$\phi$, we can reduce to the case were~$f$ is a closed embedding with~$n < m$, the case~$n=m$ being trivial. We clearly have~$C'' = 0 = C'$, so we are left to prove that:
    \begin{equation}
        \label{eq:pushforward-embedding-normal}
        T'' = \alpha_*\beta^* T + \alpha'_*\beta'^*C.
    \end{equation}
    Reducing to local computations and using the (real analytic) normal form of immersions, one can reduce to the case where~$X=\R^n$,~$Y=\R^{n+p}$ and~$f:\R^n\to\R^{n+p}$ is the inclusion along the first coordinates~$x\mapsto (x,0)$. In that case, we have~$f_*\phi = \phi \boxtimes \1_{\{0\}^p}$. We are thus left to prove that~\eqref{eq:pushforward-embedding-normal} follows from \Cref{prop:exterior-product-CFgval}. Following~\Cref{sec:exterior-product}, consider the diagrams of canonical maps
    \begin{equation*}
        \P_{\R^{n+p}} \overset{i_2}{\longleftarrow} \R^{n} \times \P_{\R^p} \overset{p_2}{\longrightarrow} \P_{\R^p}
    \end{equation*}
    and 
    \begin{equation*}
        \begin{tikzcd}
            \P_{\R^n}\times \P_{\R^p} & \hat{\P}_{\R^{n+p}} \arrow[l, "\Phi"', two heads] \arrow[r, "F", two heads]                                                             & \P_{\R^{n+p}}
        \end{tikzcd}.
    \end{equation*}
    Let us also consider the manifold~$\P_0 =\P_+(T^*_{\{0\}^p}\R^p)$ and the embedding~$i: \P_0 \hookrightarrow \P_{\R^p}$, the normal cycle~$T_0 = N(\1_{\{0\}^p}) = i_*\intcur{\P_0}$ and denote by~$\tilde{p} : \P_{\R^{n+p}}\to \R^n$ the composition of the canonical maps~$\P_{\R^{n+p}}\to \R^{n+p}\to \R^n$. Then, \Cref{prop:exterior-product-CFgval} ensures that:
    \begin{equation*}
        T'' = F_*\Phi^*(T\boxtimes T_0) + \tilde{p}^*C\cap i_{2*}p_2^*T_0.
    \end{equation*}
    The result will then follow from the equalities:
    \begin{align}
        F_*\Phi^*(T\boxtimes T_0) = \alpha_*\beta^*T, \label{eq:pushforward-immersion-1}\\
        \tilde{p}^*C\cap i_{2*}p_2^*T_0 = \alpha'_*\beta'^*C.\label{eq:pushforward-immersion-2}
    \end{align}
    To prove~\eqref{eq:pushforward-immersion-1}, consider the following commutative diagram:
    \begin{equation}\label{eq:diag-pushforward-immersion}
        \begin{tikzcd}
            \P_{\R^n}                                                                                       &                                                                                                                                                                                         & \P_{\R^n}\times\P_0 \arrow[rr, "j", hook] \arrow[ll, "q"', two heads]                        &  & \P_{\R^n}\times\P_{\R^p}                                           \\
                                                                                                            & \Int\left(\hat{\P}_0\right) \arrow[rd, "\tilde{\alpha}"'] \arrow[r, "j''", hook] \arrow[ld, "\iota", hook'] \arrow[ru, "\tilde{\Phi}", two heads] \arrow[lu, "\tilde{\beta}"', two heads] & \hat{\P}_0 \arrow[u, "\Phi_0"', two heads] \arrow[rr, "j'", hook] \arrow[d, "F_0"] &  & \hat{\P}_{\R^{n+p}} \arrow[u, "\Phi"', two heads] \arrow[lld, "F"] \\
            \tilde{\R^n\times_{\R^{n+p}}\P_{\R^{n+p}}} \arrow[uu, "\beta", two heads] \arrow[rr, "\alpha"'] &                                                                                                                                                                                         & \P_{\R^{n+p}}                                                                                &  &                                                                   
        \end{tikzcd}
    \end{equation}
    where~$j=\id_{\P_{\R^n}} \times i$, the smooth manifold with boundary $\hat{\P}_0$ and the maps $\Phi_0$ and $j'$ are defined as the pullback of the fiber bundle $\Phi:\hat{\P}_{\R^{n+p}}\twoheadrightarrow\P_{\R^n}\times\P_{\R^p}$ by $j$ and all other maps are canonical maps. More precisely, following the notations of \Cref{sec:exterior-product}, the manifold~$\hat{\P}_0$ is the closure of the set~$\{(x_1,0,[\xi_1:\xi_2],[\xi_1],[\xi_2])\}$ inside~$\P_{\R^{n+p}}\times_{\R^{n+p}} (\P_{\R^n}\times\P_{\R^p})$. Similarly, in the coordinates of~\eqref{eq:oriented-blowup-sphere} the map~$\iota$ is the obvious map:
    \begin{equation*}
        \iota(x_1,0,[\xi_1:\xi_2],[\xi_1],[\xi_2]) = (x_1,[\xi_1:\xi_2],[\xi_1]),
    \end{equation*}
    and the other maps of~\eqref{eq:diag-pushforward-immersion} are uniquely defined by the commutativity of the diagram. 

    Using that the top right square of~\eqref{eq:diag-pushforward-immersion} is Cartesian, we can compute:
    \begin{equation}\label{eq:pushforward-intermediate}
        F_*\Phi^*(T\boxtimes T_0) = F_*\Phi^*j_*q^*T        = F_*j'_*\Phi_0^*q^*T.
    \end{equation}
    Arguing as in the proof of~\Cref{prop:exterior-product-CFgval}, we have the equality of operations on currents~$\Phi_0^* = j''_*\tilde{\Phi}^*$ and~$\beta^* = \iota_*\tilde{\beta}^*$.
    Indeed, for any compactly supported smooth differential form~$\omega$ on~$\hat{\P}_0$ the equality~$\Phi_{0*}\omega =\tilde{\Phi}_*  j''^*\omega$ follows from the fact the fibers of~$\Phi_0$ are the closure of the fibers of~$\tilde{\Phi}$ in~$\hat{\P}_0$ so that the integration along the fibers of~$\Phi_0$ of the smooth form~$\omega$ and the integration along the fibers of~$\tilde{\Phi}$ of the restriction~$j''^*\omega$ of~$\omega$ to the interior~$\Int(\hat{\P}_{0})$ coincide. A similar proof holds for~$\beta^* = \iota_*\tilde{\beta}^*$. Thus~\eqref{eq:pushforward-intermediate} implies:
    \begin{align*}
        F_*\Phi^*(T\boxtimes T_0) 
        &= F_*j'_*j''_*\tilde{\Phi}^*q^*T \\
        &= \tilde{\alpha}_*\tilde{\beta}^*T \\
        &= \alpha_*\iota_*\tilde{\beta}^*T \\
        &= \alpha_*\beta^*T.
    \end{align*}

    To prove~\eqref{eq:pushforward-immersion-2}, consider the following diagram:
    \begin{equation*}
        \begin{tikzcd}
            &                                 & \R^n\times\P_0 \arrow[r, "q_2"', two heads] \arrow[d, "i'", hook] \arrow[ld, "\alpha'", hook'] \arrow[lld, "\beta'"'] & \P_0 \arrow[d, "i", hook] \\
        \R^n & \P_{\R^{n+p}} \arrow[l, "\tilde{p}"] & \R^n\times\P_{\R^p} \arrow[r, "p_2"', two heads] \arrow[l, "i_2", hook']                                               & \P_{\R^p}                 
        \end{tikzcd}
    \end{equation*}
    where the square is Cartesian up to a change of orientation by~$(-1)^{n(p-1)}$ (corresponding to exchanging the two factors). 
    Moreover, one has that
    \begin{equation*}
        q_2^*\intcur{\P_0} = (-1)^{n(p-1)}\intcur{\R^n \times\P_0},
    \end{equation*}
    so that:
    \begin{equation*}
        i_{2*}p_2^*T_0 = (-1)^{n(p-1)} i_{2*}i'_*q_2^*\intcur{\P_0} = \alpha'_*\left(\intcur{\R^n \times\P_0}\right). 
    \end{equation*}
    Therefore, we have:
    \begin{align*}
        \tilde{p}^*C\cap i_{2*}p_2^*T_0 
        &= \tilde{p}^*C \cap \alpha'_*\left(\intcur{\R^n \times\P_0}\right) \\
        &= \alpha'_*\left(\alpha'^*\tilde{p}^*C \cap \intcur{\R^n \times\P_0}\right) \\
        &= \alpha'_*\beta'^*C.
    \end{align*}

    (ii) Assume that $f$ is a submersion such that~$\bar{df}$ is angularly transverse to~$N(\phi)$. This transversality assumption ensures that $f_*\!\emCFgval[\phi]$ is well-defined by \cite[Prop.~4.5]{AB17}. Since $f$ is proper on~$\supp(\phi)$, the pushforward~$f_*\phi$ is also well-defined.
    Again, using additivity of the pushforward and the properness of~$f$ on the support of~$\phi$, we can reduce to the case where~$\phi$ is supported in an open subset where the submersion~$f$ can be put in real analytic normal form. It is thus sufficient to prove the result for~$X=\R^n$,~$Y=\R^m$ and~$f:\R^n\to\R^m$ is the projection on the first coordinates.

    Now, let us prove that $T'' = \bar{\tau}_*\bar{df}^*T$, i.e., that $N(f_*\phi) = \bar{\tau}_*\bar{df}^*N(\phi)$. To do so, we use the formula for the pushforward operation on characteristic cycles (\Cref{lem:pushforward-CC}) which ensures that~$\CC(f_*\phi) = \tau_*df^*\CC(\phi)$ as Borel-Moore homology classes. 
    Using the canonical metrics on~$X=\R^n$ and~$Y=\R^m$ (which amounts to have chosen local real analytic metrics on~$X$ and~$Y$), we can identify the projective bundles~$\P_X$ and~$\P_Y$ with unit cosphere bundles~$S^*X\subset T^*X$ and~$S^*Y\subset T^*Y$, yielding a commutative diagram:
        \begin{equation}\label{eq:cosphere-bundles}
        \begin{tikzcd}
            \P_X         \arrow[d, "\iota_X"', hook']                                                            & X\times_Y\P_Y \arrow[l, "\bar{df}"', hook'] \arrow[r, "\bar{\tau}", two heads]  \arrow[d, "\iota"', hook']                                                                                                       & \P_Y    \arrow[d, "\iota_Y"', hook']                                                              \\
            T^*X                                                                      & X\times_Y T^*Y \arrow[l, "df"', hook'] \arrow[r, "\tau", two heads]                                                                                                                    & T^*Y                                                                     
        \end{tikzcd}
    \end{equation}
    Note that the square on the right of~\eqref{eq:cosphere-bundles} is Cartesian. Therefore, we have the equalities of Borel-Moore homology classes:
    \begin{align*}
        \iota_Y^*\CC(f_*\phi) &= \iota_Y^*\tau_*df^*\CC(\phi) \\
        &= \bar{\tau}_*\iota^*df^*\CC(\phi)\\
        &= \bar{\tau}_*\bar{df}^*\iota_X^*\CC(\phi).
    \end{align*}
    Note that due to conicality, the map~$\iota_Y$ is transverse to the current~$\CC(f_*\phi)$ and~$\iota_X$ to $\CC(\phi)$. Moreover, our transversality assumption ensures that the map~$\bar{df}$ is transverse to the current~$\iota_X^*\CC(\phi)=N(\phi)$, so that \Cref{lem:cur-BM-pullback,lem:cur-BM-pushforward} ensure that we have~$\iota_Y^*\CC(f_*\phi) = \bar{\tau}_*\bar{df}^*\iota_X^*\CC(\phi)$ as subanalytic currents, that is, $N(f_*\phi) = \bar{\tau}_*\bar{df}^*N(\phi)$.

    Therefore, the generalized valuation~$f_*\emCFgval[\phi] - \emCFgval[f_*\phi] \in\gval[Y]$ is represented by the pair of currents~$(C'-C'',0)$. By additivity, we can assume without loss of generality that~$X$ is connected, so that~\Cref{lemma_t_equals_0} ensures~$f_*\emCFgval[\phi] - \emCFgval[f_*\phi] = \lambda \chi$ for some~$\lambda \in \C$. We are thus left to prove that~$\lambda = 0$. It follows from the definitions that the support of~$f_*\emCFgval[\phi] - \emCFgval[f_*\phi]$ is included in~$f(\supp(\phi))$; see e.g.~\cite[Eq.~(28)]{AB09}. Therefore, if~$f(\supp(\phi)) \ne Y$, then we must have~$\lambda = 0$. More generally, if~$\dim(Y) > 0$, we can use \cite[Prop.~6.2.1]{alesker_val_man4} to write $\phi = \phi_1 + \phi_2$ with $f(\supp(\phi_j)) \ne Y$ for~$j=1,2$. By additivity of pushforward operations, we can reduce to the previous case and prove again that~$\lambda=0$. If~$\dim(Y) = 0$, we can assume~$Y = \{\pt\}$ without loss of generality so that~$f_*\phi$ is constant equal to~$\mu = \int_X\phi\dEuler$, and hence~$\emCFgval[f_*\phi] = \mu\cdot\chi \in\gval[Y]$. Similarly, we have~$f_*\emCFgval[\phi] = \mu\cdot\chi$ by~\cite[Prop.~8.3.1]{alesker_val_man4}, so~$f_*\emCFgval[\phi] = \emCFgval[f_*\phi]$ once again.
\end{proof}

\subsection{Convolution}\label{sec:convolution}
We recall the definition of the convolution of generalized valuations introduced in~\cite{AB17}. Let~$G$ be a Lie group acting transitively on a smooth manifold~$X$. As explained in the introduction, the manifolds~$G$ and~$X$ are then equipped with canonical real analytic structures and the action~$a:G\times X \to X$ is real analytic. The \emph{convolution} of the generalized valuations~$\mu\in\gval[G]$ and~$\zeta\in\gval[X]$ is defined as~$\mu*\zeta := a_*(\mu\boxtimes\zeta)\in\gval[X]$, provided that the pushforward on the right-hand side is well-defined. 

We can now prove our result on the restriction of the convolution of generalized valuations to constructible functions.
\begin{proof}[Proof of~\Cref{cor:convolution-CFgval}]
    The convolution of the constructible functions~$\phi\in\CF[][G]$ and~$\psi\in\CF[][X]$ is defined as soon as the action~$a:G\times X \to X$ is proper on the support of $\phi\boxtimes\psi$, which is assumed in~(ii) of~\Cref{main:pushforward}. Therefore, under this assumption, the convolution~$\emCFgval[\phi]*\emCFgval[\psi]$ is well-defined in the sense of generalized valuations and so is $\phi * \psi$ in the sense of constructible functions. The result then directly follows from Alesker's result on the exterior product~(\Cref{prop:exterior-product-CFgval}) and our result on the pushforward~(\Cref{main:pushforward}).
\end{proof}

\section{Additive and multiplicative kinematic formulas}\label{sec:kinematic-formulas}

In this section we prove additive and multiplicative kinematic formulas for constructible functions. As we have noted earlier, the convolution product of constructible functions is a good replacement for the Minkowski sum of convex bodies. We will prove such formulas in two different cases. 

The first case is on a Euclidean vector space acting on itself by addition. We consider valuations that are translation-invariant and invariant with respect to a subgroup of the orthogonal group that acts transitively on the unit sphere. In this case, additive kinematic formulas for convex bodies are well-known. To prove their extension to the setting of constructible functions we first prove a statement of independent interest (Theorem \ref{thm_compatibility_convolution_generalized}) that relates compactly supported generalized valuations, translation-invariant generalized valuations and their respective convolution products. It generalizes results of the first named author with Alesker in the smooth case \cite{AB17} and with Faifman in the polytopal case \cite{bernig_faifman16}. 

The second case is that of constructible functions on the Lie group $S^3$. The convolution product of two geodesically convex subsets of $S^3$ is in general not geodesically convex. Hence, there are no multiplicative kinematic formulas for geodesically convex sets and we are forced to state our formulas with constructible functions.  

In both cases the convolution of compactly supported constructible functions is well-defined. However, the convolution product of the associated generalized valuations is only defined under some conditions. We will show that these conditions are satisfied for constructible functions in generic position.

\subsection{Compactly supported versus translation-invariant valuations}

In order to distinguish between $\R^n$ and its dual, it will be convenient to write~$V:=\R^n$. 

\begin{defi} \label{def_transversal_compact}
Two compactly supported generalized valuations~$\zeta_1,\zeta_2\in\mathcal{V}^{-\infty}_\cpt(V)$ are \emph{transverse} if there are closed conical sets $\Lambda_i \subset T^*V \setminus \underline 0$ and~$\Gamma_i \subset T^*\P_V \setminus \underline 0$ with~$\WF(\zeta_i) \subset (\Lambda_i,\Gamma_i)$ such that the following condition holds:
\begin{itemize}
    \item[] There are no $(x_i,[\xi]) \in \P_V, i=1,2$ and $\eta \in T^*_{[\xi]} \P_+(V^*) \subset T^*_{(x_i,[\xi])}\P_V$  such that $(x_1,[\xi],\eta) \in \Gamma_1$ and $(x_2,[\xi],-\eta) \in \Gamma_2$. 
\end{itemize}	 
\end{defi}

\begin{prop} \label{prop_transversal_compact_valuations}
If $\zeta_1$ and $\zeta_2$ are transverse, then the convolution $\zeta_1 * \zeta_2$ is well-defined. More precisely, denoting
$(\Lambda_i,\Gamma_i)$ as in the previous definition, the convolution on compactly supported smooth valuations on $V$ extends as a jointly sequentially continuous map 
\begin{displaymath}
    \mathcal V^{-\infty}_{\cpt,\Lambda_1,\Gamma_1}(V) \times \mathcal V^{-\infty}_{\cpt,\Lambda_2,\Gamma_2}(V) \to \mathcal V^{-\infty}_\cpt(V).
\end{displaymath}
\end{prop}

\proof 
With $(\Lambda,\Gamma)$ as in \cite[Prop.~4.1]{AB17}, the exterior product extends to a jointly sequentially continuous map 
\begin{displaymath}
\mathcal V^{-\infty}_{\cpt,\Lambda_1,\Gamma_1}(V) \times \mathcal V^{-\infty}_{\cpt,\Lambda_2,\Gamma_2}(V) \to \mathcal V^{-\infty}_{\cpt,\Lambda,\Gamma}(V \times V).
\end{displaymath}
Moreover, the pushforward extends to a sequentially continuous map 
\begin{displaymath}
\mathcal V^{-\infty}_{\cpt,\Lambda,\Gamma}(V \times V) \to \mathcal V^{-\infty}_\cpt(V),
\end{displaymath}
provided that $\Gamma$ is disjoint from $T^*_{(V \times V) \times_V \P_V}\P_{V \times V}$ \cite[Prop.~4.5]{AB17}. To check this last condition, notice that the submanifold~$(V \times V) \times_V \P_V$ consists of all points~$(x_1,x_2,[\xi:\xi]) \in \P_{V \times V}$. Then, it is easily checked that 
\begin{displaymath}
\Gamma_{(x_1,x_2,[\xi:\xi])}=\{(\rho_1,\rho_2) : \rho_i \in \Gamma_i|_{(x_i,[\xi])}\},
\end{displaymath}
where we use the injection
\begin{displaymath}
	T^*_{(x_1,[\xi])}\P_V \oplus T^*_{(x_2,[\xi])}\P_V \hookrightarrow T^*_{(x_1,x_2,[\xi:\xi])}\P_{V \times V}
\end{displaymath}
that is dual to the projection 
\begin{displaymath}
	d(\Phi \circ F^{-1}): T_{(x_1,x_2,[\xi:\xi])}\P_{V \times V} \twoheadrightarrow
	T_{(x_1,[\xi])}\P_V \oplus T_{(x_2,[\xi])}\P_V 
\end{displaymath}
from \cite[Section 4.1]{AB17}.

Let us write $\rho_i=\rho_i'+\rho_i''$ according to the decomposition
\begin{displaymath}
T^*_{(x_i,[\xi])} \P_V = T^*_{x_i}V \oplus T^*_{[\xi]} \P_+(V^*).
\end{displaymath}
Then $(\rho_1,\rho_2)$ vanishes on $T_{(x_1,x_2,[\xi:\xi])} (V \times V) \times_V \P_V$ if and only if $\rho_1'=\rho_2'=0$ and $\rho_1''=-\rho_2''$. This will never be the case when~$(\rho_1,\rho_2)\in\Gamma_{(x_1,x_2,[\xi:\xi])}$ by our transversality assumptions, hence the result.
\endproof

We want to compare the convolution of compactly supported valuations with the convolution of translation-invariant valuations and have to recall some definitions first. We refer to \cite{alesker_introduction, schneider_book14} for more information.  

The set of convex bodies in $V$ will be denoted by $\mathcal K^n$. A valuation on $V$ is a map $\mu:\mathcal K^n \to \R$ that satisfies
\begin{displaymath}
\mu(K \cup L)+\mu(K \cap L)=\mu(K)+\mu(L)
\end{displaymath}
whenever $K,L, K \cup L \in \mathcal K^n$. The vector space of translation-invariant and continuous valuations is denoted by $\Val$. It admits a decomposition $\Val=\bigoplus_{k=0}^n \Val_k$, where $\Val_k$ denotes the subspace of $k$-homogeneous elements. 

The space $\Val$ is a Banach space with a natural action of the group $\mathrm{GL}(n)$. The smooth vectors of this action form a dense subspace $\Val^\infty$ that admits a Fr\'echet space topology. 

A smooth translation-invariant valuation $\mu\in\Val^\infty$ can be written as 
\begin{displaymath}
\mu(K)=\int_K \oldphi+\int_{N(K)} \omega,
\end{displaymath} 
where $N(K)$ is the normal cycle of the convex body $K$ (see for instance~\cite{RZ19}) and $\oldphi \in \Omega^n(V)^{tr}$ and $\omega \in \Omega^{n-1}(\P_V)^{tr}$ are translation-invariant forms. Comparing with \eqref{eq:def-val} it is easy to show that $\Val^{\infty}(V)=(\mathcal V^{\infty}(V))^{tr}$, i.e. that the two notions of ``smooth and translation-invariant valuations'' coincide. 

The convolution of smooth translation-invariant valuations is a commutative and associative product on $\Val^\infty \otimes \Dens(V^*)$ that is characterized by the property
\begin{displaymath}
[\vol(\bullet+A) \otimes \vol^*] * 	[\vol(\bullet+B) \otimes \vol^*]=	[\vol(\bullet+A+B) \otimes \vol^*]
\end{displaymath}
for all convex bodies $A,B$ with smooth boundary of positive curvature.

Alesker, using the product of smooth translation-invariant valuations, constructed in \cite[Thm.~6.1.1]{alesker_val_man4} an injective map 
\begin{displaymath}
\PD:\Val^\infty(V) \hookrightarrow (\Val^\infty(V))^* \otimes \Dens(V)=:\Val^{-\infty}(V)
\end{displaymath}
with dense image. Elements on the right-hand side are called translation-invariant generalized valuations. Similar to \Cref{prop:gval-as-cur} they can be uniquely represented by a pair of currents~$(C,T)$, where $C$ is a translation-invariant $n$-current on $V$ (i.e. a multiple of the integration current) and $T$ is a translation-invariant Legendrian $(n-1)$-cycle on $\P_V$. The wave front set of a valuation is defined as the wave front set of $T$, which can be represented as a closed conical subset of $T^*\P_+(V^*) \setminus \underline 0$ by translation-invariance. The subset of $\Val^{-\infty}(V)$ consisting of valuations with wave front set contained in some closed conical set $\Gamma \subset T^*\P_+(V^*) \setminus \underline 0$ is denoted by $\Val_\Gamma^{-\infty}(V)$. It is endowed with the H\"ormander topology. 

\begin{defi} \label{def_transversal_translation}
Two elements $\mu_1, \mu_2 \in \Val^{-\infty}(V) \otimes \Dens(V^*)$ are \emph{transverse} if there are closed conical sets $\Gamma_i \subset T^* \P_+(V^*) \setminus \underline 0$ with $\WF(\mu_i) \subset \Gamma_i$ such that~$\Gamma_1 \cap s(\Gamma_2) = \emptyset$ where~$s:T^* \P_+(V^*)\to T^* \P_+(V^*)$ is the multiplication by~$-1$ in the fibers.
\end{defi} 

\begin{prop}[{\cite[Prop.~4.7]{bernig_faifman16}}] \label{prop_continuity_gen_translationinv}
If $\mu_1,\mu_2\in \Val^{-\infty}(V) \otimes \Dens(V^*)$ are transverse, then the convolution product $\mu_1 * \mu_2$ is defined. More precisely, if $\Gamma_1,\Gamma_2$ satisfy the condition of the previous definition, then the convolution extends as a jointly sequentially continuous map 
\begin{displaymath}
    \left(\Val_{\Gamma_1}^{-\infty}(V) \otimes \Dens(V^*)\right) \times \left(\Val_{\Gamma_2}^{-\infty}(V) \otimes \Dens(V^*)\right) \to \Val^{-\infty}(V) \otimes \Dens(V^*).
\end{displaymath}
\end{prop}

To motivate the next theorem, let us recall two results related to convolution of valuations. We denote by $\vol$ any choice of Lebesgue measure on $V$, and $\vol^*$ the dual Lebesgue measure on $V^*$, so that $\vol \otimes \vol^*$ is independent of the choices.
\begin{thm}[{\cite[Thm.~2]{AB17}}] \label{thm_ab_average}
Define a map 
\begin{align*}
    \tilde F:\mathcal V_\cpt^\infty(V) & \to \mathcal V^\infty(V)^{tr} \otimes \Dens(V^*), \\
    \zeta & \mapsto \int_V \zeta(\bullet-x) d\vol(x) \otimes \vol^*
\end{align*} 
and let 
\begin{displaymath}
F:\mathcal V_\cpt^\infty(V) \to \Val^\infty(V) \otimes \Dens(V^*)
\end{displaymath}
be the composition of $\tilde F$ with the natural isomorphism $\mathcal V^\infty(V)^{tr} \otimes \Dens(V^*) \cong \Val^\infty(V) \otimes \Dens(V^*)$. Then $F$ is a surjective homomorphism with respect to the convolution products on both sides. 
\end{thm}

The polytope algebra~$\Pi(V)$ is the algebra generated by all symbols~$\hat P$, where~$P$ is a polytope, subject to the relations~$\hat P+\hat Q=\widehat{P \cup Q}+\widehat{P \cap Q}$ whenever~$P,Q,P \cup Q$ are polytopes, and~$\widehat{P+x}=\hat P$ for all~$x \in V$. The product is defined on generators by~$\hat P \cdot \hat Q:=\widehat{P+Q}$, see \cite{mcmullen_polytope_algebra}. Note that the indicator function of a polytope is a constructible function, hence it defines a generalized valuation~$[\mathbf 1_P]$. This is why we deviate from the standard notation which uses~$[P]$ instead of~$\hat P$.

\begin{thm}[{\cite[Thm.~2.5]{bernig_faifman16}}] \label{thm_bf_average}
The map on polytopes 
\begin{displaymath}
M(P):=[\mu \mapsto \mu(P)] \in \Val^{\infty}(V)^* \cong\Val^{-\infty}(V) \otimes \Dens(V^*)
\end{displaymath} 
extends to an injective map $M: \Pi(V) \to \Val^{-\infty}(V) \otimes \Dens(V^*)$. Moreover, if~$x,y \in \Pi(V)$ are in general position, then~$M(x)$ and~$M(y)$ are transverse and~$M(x \cdot y)=M(x) * M(y)$. 
\end{thm}

In the following we will simply write $\mu(\zeta)$ for the natural pairing between~$\mu \in \Val^\infty(V) \subset \mathcal V^\infty(V)$ and~$\zeta \in \mathcal V_\cpt^{-\infty}(V)$. We define the map 
\begin{align*}
M:\mathcal V^{-\infty}_\cpt(V) & \to \Val^{-\infty}(V) \otimes \Dens(V^*) \cong \Val^\infty(V)^*, \\
\zeta & \mapsto \left[\mu \mapsto \mu(\zeta), \mu \in \Val^\infty(V)\right].
\end{align*} 
Note that for polytopes, we have $M([\1_P])=M(P)$, where the right hand side refers to the map from Theorem \ref{thm_bf_average}, which justifies the use of the same letter. 

Next, we define a map 
\begin{align*}
\tilde M:\mathcal V^{-\infty}_\cpt(V) & \to \mathcal V^{-\infty}(V)^{tr} \otimes \Dens(V^*), \\
\zeta & \mapsto \int_V \zeta(\bullet-x) d\vol(x) \otimes \vol^*.
\end{align*} 
Here $\zeta(\bullet-x) \in \mathcal V_\cpt^{-\infty}(V)$ is defined as $(t_x)_*\zeta$ for the translation $t_x:V \to V$. 

\begin{lem}
The transpose $F^*$ of the map $F$ from Theorem \ref{thm_ab_average} is an isomorphism. The following diagram commutes 
\begin{equation*}
\begin{tikzcd}
    \mathcal V_\cpt^{-\infty}(V) \arrow[r,"M"] \arrow[rr,bend left=20,"\tilde M"] & \Val^{-\infty}(V) \otimes \Dens(V^*) \arrow[r,"F^*","\cong"'] & \mathcal V^{-\infty}(V)^{tr} \otimes \Dens(V^*)    \\
    \mathcal V_\cpt^{\infty}(V) \arrow[r,"F"] \arrow[rr,,bend right=20,"\tilde F"] \arrow[u,hook] & \Val^{\infty}(V) \otimes \Dens(V^*)  \arrow[r,"\cong"] \arrow[u,hook] & \mathcal V^\infty(V)^{tr} \otimes \Dens(V^*) \arrow[u,hook]                               
\end{tikzcd}.
\end{equation*}
\end{lem}

\proof
The statement about $F^*$, as well as the commutativity of the right hand square is proved in \cite[Prop.~2.5]{bernig_faifman16}. 

Let us show that $\tilde M=F^* \circ M$. For polytopes this was shown in \cite[Lemma 3.2]{bernig_faifman16} and the general case follows the same lines. Let~$\zeta \in \mathcal V^{-\infty}_\cpt(V)$ and~$\beta \in \mathcal V^\infty_\cpt(V)$. Then 
\begin{align*}
\langle F^* \circ M(\zeta),\beta\rangle & = \langle M(\zeta),F(\beta)\rangle\\
& = F(\beta)(\zeta)\\
& = \tilde F(\beta)(\zeta)\\
& = \int_V \langle \beta(\bullet-x),\zeta\rangle \d\!\vol(x) \otimes \vol^*\\
& = \int_V \langle \beta,\zeta(\bullet+x)\rangle \d\!\vol(x) \otimes \vol^*\\
& = \int_V \langle \beta,\zeta(\bullet-x)\rangle \d\!\vol(x) \otimes \vol^*\\
& = \langle \tilde M(\zeta),\beta\rangle.
\end{align*}

It is clear from the definitions that the restriction of $\tilde M$ to $\mathcal V_\cpt^\infty(V)$ equals~$\tilde F$, which implies the commutativity of the left hand square. 

\endproof

\begin{lem} \label{lemma_wavefront_of_average}
Let $(\Lambda_1,\Gamma_1)$ with $\Lambda_1 \subset T^*V \setminus \underline 0, \Gamma_1 \subset T^*\P_V \setminus \underline 0$. We set \begin{equation*}
\hat \Gamma:=\left\{([\xi],\eta'') \in T^*\P_+(V^*) : \exists \,\tilde x \in V \mbox{ such that } (0,\eta'') \in \Gamma_1|_{(\tilde x,[\xi])}\right\}.
\end{equation*}
Then, for $\zeta \in \mathcal V_{\cpt,\Lambda_1,\Gamma_1}^{-\infty}(V)$ one has $\WF(M(\zeta)) \subset \hat \Gamma$ and the map 
\begin{displaymath}
M: \mathcal V_{\cpt,\Lambda_1,\Gamma_1}^{-\infty}(V)  \to \Val_{\hat \Gamma}^{-\infty}(V)  \otimes \Dens(V^*)
\end{displaymath} 
is sequentially continuous. 
\end{lem}

\proof 
To simplify the notation, we fix a Lebesgue measure on $V$ and use it to identify $\Dens(V^*) \cong \C$.
Let $(C_1,T_1)$ be the pair of currents corresponding to $\zeta$. In particular, we have~$\WF(T_1) \subset \Gamma_1$.  
Let $\bar \tau':\P_V \times V \to \P_V$ be the map $(x,[\xi],y) \mapsto (x+y,[\xi])$ and set 
\begin{displaymath}
    T:=(\bar \tau')_*(T_1 \boxtimes \vol).
\end{displaymath}	

Fix $(x,[\xi],y) \in \P_V \times V$ and $v \in T_{(x,[\xi])}\P_V$. Then $(v,0) \in T_{(x,[\xi],y)}(\P_V \times V)$ and $d\bar \tau'(v,0)=dt_{y}(v)$, where $t_y:\P_V \to \P_V$ is translation by $y$.

Let $\omega$ be a compactly supported $(n-1)$-form on $\P_V$. We can consider $t_y^*\omega$ as an $(n-1)$-form on $\P_V \times V$. More precisely, its value at the point $(x,[\xi],y)$ is given by $(t_y^*\omega)|_{(x,[\xi])} \in \wedge^{n-1}T^*_{(x,[\xi])}\P_V \subset \wedge^{n-1}T^*_{(x,[\xi],y)}(\P_V \times V)$. 

It follows from the above that $(\bar \tau')^* \omega|_{(x,[\xi],y)}=t_y^*\omega|_{(x,[\xi])}$ modulo the ideal generated by the forms $dy^j$ for~$j=1,\ldots,n$. Since $T_1 \boxtimes \vol$ vanishes on this ideal, we obtain  
\begin{displaymath}
    \langle T,\omega\rangle=\langle T_1 \boxtimes \vol,(\bar \tau')^* \omega\rangle=\left\langle T_1,\int_V t_y^* \omega dy\right\rangle=\left\langle \int_V (t_y)_* T_1 dy,\omega\right\rangle.
\end{displaymath}
Therefore, the current $T$ is the Legendrian cycle corresponding to the generalized valuation~$\tilde M(\zeta)=\int_V t_{y*}\zeta dy$. 

With $\vol$ being a smooth current, \cite[Prop.~2.16]{brouder_dang_helein} implies that  
\begin{displaymath}
\WF(T_1 \boxtimes \vol)|_{(x,[\xi],y)}=\{(\kappa,0): \kappa \in \WF(T_1)|_{(x,[\xi])}\} \subset \{(\kappa,0): \kappa \in \Gamma_1|_{(x,[\xi])}\}.
\end{displaymath}
We can write $\kappa=(\kappa',\kappa'')$ with $\kappa' \in T_x^*V$ and~$\kappa'' \in T_{[\xi]}^*\P_+(V^*)$. It follows from \cite[Prop.~2.15]{brouder_dang_helein} that
\begin{align*}
\WF(T)|_{(x,[\xi])} 
&\subset \Big\{(\eta',\eta'') \in T^*_{(x,[\xi])}\P_V: \exists\, \tilde x\in V, (\kappa',\kappa'') \in \Gamma_1|_{(\tilde x,[\xi])}, \\
    &\hspace{3cm} (d\bar \tau'|_{(\tilde x,[\xi],x-\tilde x)})^*(\eta',\eta'')=(\kappa',\kappa'',0)\Big\}.
\end{align*}

Since $(d\bar \tau'|_{(\tilde x,[\xi],x-\tilde x)})^*(\eta',\eta'')=(\eta',\eta'',\eta')$, the last set is precisely the image of $\hat \Gamma$ under the embedding $T^*\P_+(V^*) \setminus \underline 0 \hookrightarrow T^*\P_V \setminus \underline 0$ and hence the wave front set of $M(\zeta)$ is contained in $\hat \Gamma$. Since exterior product and pushforward of currents are sequentially continuous operations with respect to the corresponding H\"ormander topologies \cite[Props.~2.15 and~2.16]{brouder_dang_helein}, the sequential continuity of $M$ follows.
\endproof

The next theorem generalizes Theorems \ref{thm_ab_average} and \ref{thm_bf_average}. 

\begin{thm} \label{thm_compatibility_convolution_generalized}
If $\zeta_1, \zeta_2\in\mathcal{V}^{-\infty}_\cpt(V)$ are transverse, then $M(\zeta_1)$ and~$M(\zeta_2)$ are transverse as well and  
\begin{displaymath}
    M(\zeta_1 * \zeta_2)=M(\zeta_1) * M(\zeta_2).
\end{displaymath}
\end{thm}

\proof 
Let $\Lambda_i \subset T^*V \setminus \underline 0$ and $\Gamma_i \subset T^*\P_V \setminus \underline 0$ be as in Definition \ref{def_transversal_compact}. Let~$\hat \Gamma_i \subset T^*\P_+(V^*) \setminus \underline 0$ be the corresponding sets defined in~\Cref{lemma_wavefront_of_average}. Clearly~$\hat \Gamma_1$ and~$\hat \Gamma_2$ satisfy the condition from Definition \ref{def_transversal_translation}. This shows that~$M(\zeta_1)$ and~$M(\zeta_2)$ are transverse.  

Using \cite[Lemma 8.2]{AB09} we can approximate $\zeta_j$ in $\mathcal V^{-\infty}_{\cpt,\Lambda_j,\Gamma_j}(V)$ by a sequence of smooth valuations $\zeta_j^i \in \mathcal V^\infty_\cpt(V)$ for~$j=1,2$.

By Lemma \ref{lemma_wavefront_of_average}, $M(\zeta_j^i)$ converges to $M(\zeta_j)$ in $\Val^{-\infty}_{\hat \Gamma_j}(V) \otimes \Dens(V^*)$ and hence \Cref{prop_continuity_gen_translationinv} implies that $M(\zeta_1^i) * M(\zeta_2^i)$ converges to $M(\zeta_1) * M(\zeta_2)$ in $\Val^{-\infty}(V) \otimes \Dens(V^*)$. 

On the other hand, $\zeta_1^i * \zeta_2^i$ converges to $\zeta_1 * \zeta_2$ in $\mathcal V_\cpt^{-\infty}(V)$ by \Cref{prop_transversal_compact_valuations}. \Cref{lemma_wavefront_of_average} (with $\Gamma_1=T^*\P_V \setminus \underline 0$) then implies that~$M(\zeta_1^i * \zeta_2^i)$ converges to~$M(\zeta_1 * \zeta_2)$ in~$\Val^{-\infty}(V) \otimes \Dens(V^*)$.  

Since $M(\zeta_1^i) * M(\zeta_2^i)=F(\zeta_1^i) * F(\zeta_2^i)=F(\zeta_1^i * \zeta_2^i)=M(\zeta_1^i * \zeta_2^i)$ by \Cref{thm_ab_average}, the proof is complete.
\endproof

\subsection{Constructible functions in general position}

The aim of this subsection is to prove the following result. 

\begin{prop} \label{prop_convolution_constructible_transitive}
Let $G \subset \mathrm{O}(n)$ be a Lie subgroup that acts transitively on the unit sphere and $\phi_1,\phi_2 \in \mathrm{CF}(\R^n)$ be compactly supported. Then for almost all $(g_1,g_2) \in G \times G$, the constructible functions $(g_1)_*\phi_1$ and $(g_2)_*\phi_2$ satisfy the assumptions of  Corollary \ref{cor:convolution-CFgval}. In particular, the generalized valuations $[(g_1)_*\phi_1]$ and $[(g_2)_*\phi_2]$ are transverse in the sense of \Cref{def_transversal_compact}. 
\end{prop}

\proof
We want to show that for almost all $(g_1,g_2)$ the normal cycle of $N((g_1)_*\phi_1 \boxtimes (g_2)_*\phi_2)$ is transversal to the differential of the addition map, i.e. to the set $X \times_Y \P_Y$, where $X=\R^n \times \R^n$ and $Y=\R^n$. Note that $X \times_Y \P_Y \subset \P_X \setminus (\mathcal M_1 \cup \mathcal M_2)$ consists of all tuples~$(x_1,x_2,[\eta:\eta]) \in \P_X$ with~$x_1,x_2 \in \R^n$ and~$\eta \in T_{x_1+x_2}^*\R^n \setminus \{0\}$.

First we construct an angular stratification of the normal cycle. Take angular stratifications of $N(\phi_1)$ and $N(\phi_2)$. Then a stratification of the normal cycle~$N(\phi_1 \boxtimes \phi_2)$ is given by all $\hat S$ of the form 
\begin{equation} \label{eq_stratum_hatS}
\hat S:=\{(x_1,x_2,[\xi_1:\lambda \xi_2]) : (x_i,[\xi_i]) \in S_i \mbox{ for } i=1,2, \lambda \in \R_+\},
\end{equation}  
where $S_i$ is a stratum of $N(\phi_i)$. 

Similarly, a stratification of $N((g_1)_*\phi_1 \boxtimes (g_2)_*\phi_2)$ is given by all strata of the form
\begin{equation} \label{eq_stratum_ext_prod}
\{(x_1,x_2,[\xi_1:\lambda \xi_2]), (x_i,[\xi_i]) \in (g_i,g_i)S_i \mbox{ for } i=1,2, \lambda \in \R_+\},
\end{equation}  

Using the Euclidean metric, we can identify~$\P_Y$ with~$\R^n \times S^{n-1}$, by sending~$(x,v) \in \R^n \times S^{n-1}$ to~$(x,[\xi])$ where~$\xi=\langle v,\bullet\rangle$. We can identify~$\P_X \setminus (\mathcal M_1 \cup \mathcal M_2)$ with~$\R^n \times \R^n \times S^{n-1} \times S^{n-1} \times \left(0,\frac{\pi}{2}\right)$ by mapping~$(x_1,x_2,v_1,v_2,\phi)$ to~$(x_1,x_2,[\cos(\phi) \xi_1:\sin(\phi)\xi_2])\in\P_X \setminus (\mathcal M_1 \cup \mathcal M_2)$, where~$\xi_i=\langle v_i,\bullet\rangle$. 

Up to this identification, our stratification is the product of the angular stratification of $\phi_1$, the one of $\phi_2$ and the obvious stratification of the open interval $\left(0,\frac{\pi}{2}\right)$, hence it is angular by Lemma \ref{lemma_product_angular_stratifications}. Strictly speaking, since the interval is open, we only obtain that the stratification is angular outside any open neighborhood of $\mathcal M_1 \cup \mathcal M_2$, which is clearly sufficient for our purpose.

Fix $S_1,S_2$ and $\hat S$ as above. We claim that the map 
\begin{align*}
\Xi: \hat S \times G \times G &\to \P_X\\
    (x_1,x_2,[\xi_1:\lambda \xi_2],g_1,g_2) &\mapsto (g_1x_1,g_2x_2,[g_1\xi_1:\lambda g_2\xi_2])
\end{align*}
is transversal to $X \times_Y \P_Y$. 

For this, we use our identifications and consider a point
\begin{equation*}
    p=(x_1,x_2,v_1,v_2,\phi,g_1,g_2) \in \Xi^{-1}(X \times_Y \P_Y)
\end{equation*}
and let~$q=\Xi(p)=(g_1x_1,g_2x_2,v,v,\phi)$ with~$v=g_1v_1=g_2v_2$. Then~$T_{x_1}\R^n$ and~$T_{x_2}\R^n$ are contained in~$T_q (X \times_Y \P_Y)$. By definition of~$\hat S$, the subspace~$T_\phi \left(0,\frac{\pi}{2}\right)$ belongs to the image of~$d\Xi|_p$. By our assumption on~$G$, the map~$g \in G \mapsto gv\in S^{n-1}$ is a submersion. Therefore the image of~$d\Xi|_p$ contains a vector of the form~$(*,0,u,0,0)$ and a vector of the form~$(0,*,0,u,0)$ for each~$u \in T_vS^{n-1}$. This proves the claim. 

By Thom's transversality theorem \cite{GP10}, it follows that for almost all $(g_1,g_2)$ the image of $\Xi(\bullet,g_1,g_2)$ is transversal to $X \times_Y \P_Y$. This image is precisely the set \eqref{eq_stratum_ext_prod}. Doing this for all strata, we obtain that for almost all $(g_1,g_2)$, all strata of $N((g_1)_*\phi_1 \boxtimes (g_2)_* \phi_2)$ are transversal to $X \times_Y \P_Y$.

Given such a pair~$(g_1,g_2)$, we set~$\psi_i:=(g_i)_*\phi_i \in \CF[\cpt][\R^n]$. Then~$\psi_1$ and~$\psi_2$ satisfy the assumption of Corollary \ref{cor:convolution-CFgval}. Let us show that $\psi_1$ and $\psi_2$ are also transversal in the sense of Definition \ref{def_transversal_compact}.

Let~$\Gamma_i$ be the wave front set of~$N(\psi_i)$. Suppose that there are~$(x_1,[\xi],\eta) \in \Gamma_1$ and~$(x_2,[\xi],-\eta) \in \Gamma_2$ with~$\eta \in T^*_{[\xi]} \P_+(V^*) \subset T^*_{(x_i,[\xi])} \P_V$. Let~$S_i$ be the stratum of~$N(\psi_i)$ that contains~$(x_i,[\xi])$. By angularity of the stratifications, $\eta$ vanishes on~$T_{(x_i,[\xi])} S_i$. The stratum of~$N(\psi_1 \boxtimes \psi_2)$ that contains~$(x_1,x_2,[\xi:\xi])$ is~$\hat S$ from \eqref{eq_stratum_hatS}. Since~$(\eta,-\eta)$ vanishes on~$T_{(x_1,x_2,[\xi:\xi])} X \times_Y \P_Y$ and on~$T_{(x_1,x_2,[\xi:\xi])} \hat S$, we get a contradiction to the transversality of~$X \times_Y \P_Y$ and~$\hat S$.
\endproof

\begin{rk} 	
It is clear that the kernel of the map $M$ contains~$\phi-t_x^*\phi$ for all~$x \in V$. In view of the results on the polytope algebra mentioned above, it is tempting to conjecture that the kernel of $M$ is generated by such elements. We could then consider the algebra of constructible functions modulo translations and inject it into $\Val^{-\infty}(V) \otimes \Dens(V^*)$. 

However, the conjecture is not true, as the following easy example shows. Take $n=2$ and let $\phi_1=2 \cdot \mathbf{1}_{S^1}$ be twice the indicator function of the unit circle and $\phi_2= \mathbf{1}_{S^1_2}$ the indicator function of the circle with radius $2$. If $\mu$ is a smooth translation invariant valuation on $V$, then we can decompose it into homogeneous components $\mu=\mu_0+\mu_1+\mu_2$, with $\mu_0$ a multiple of the Euler characteristic and $\mu_2$ a multiple of the Lebesgue measure. Clearly Euler characteristic and Lebesgue measure both vanish on $\phi_1$ and $\phi_2$. Since~$\mu_1$ is $1$-homogeneous, we have $\mu_1(\phi_2)=2\mu_1(S^1)=\mu_1(\phi_1)$. Hence $\phi_1-\phi_2$ is in the kernel of $M$. On the other hand, $\phi_1-\phi_2$ can not be written as a finite sum of functions of the form $\phi-t_x^*\phi$, since the curvature of each arc of $S^1$ is $1$, while the curvature of each arc of $S^1_2$ is $\frac12$. 
\end{rk}

\subsection{Additive kinematic formulas}

In this section, we assume that~$G$ is a closed subgroup of $\mathrm O(n)$ that acts transitively on the unit sphere. Using the Lebesgue measure of $\R^n$ we can identify $\Dens(V) \cong \C$. The connected groups acting effectively and transitively on the unit sphere were classified by Montgomery-Samelson and Borel \cite{borel49, montgomery_samelson43}: 
\begin{gather*}
\mathrm{SO}(n),\ \mathrm{U}(n/2),\ \mathrm{SU}(n/2),\ \mathrm{Sp}(n/4),\ \mathrm{Sp}(n/4) \cdot \mathrm{U}(1),\\
    \mathrm{Sp}(n/4) \cdot \mathrm{Sp}(1),\ \mathrm{G}_2,\ \mathrm{Spin}(7),\ \mathrm{Spin}(9).
\end{gather*}
It was observed by Alesker \cite{alesker_survey07} that the vector space $\Val^G$ of continuous, translation-invariant and $G$-invariant valuations on $V$ is finite-dimensional and that all such valuations are smooth. 

Given a basis $\mu_1,\ldots,\mu_N$ of $\Val^G$, there are additive kinematic formulas:
\begin{equation} \label{eq_additive_kinematic_convex}
\int_G \mu_i(K+gL)dg=\sum_{k,l}c_{k,l}^i \mu_k(K)\mu_l(L), \quad \mbox{ for all } K,L\in\mathcal K^n.
\end{equation}

It is a difficult problem to determine such a basis and to compute the constants $c_{k,l}^i$ explicitly. In the case $G=\mathrm O(n)$ this can be achieved using Hadwiger's theorem and the template method (i.e. computing the constants by plugging in enough different examples). For the other known cases from the list above, in particular for $G=\mathrm U(n/2)$, the following link between the additive kinematic formulas and the convolution product was the key ingredient. 

Let us define a map $a:\Val^G \to \Val^G \otimes \Val^G$ by $a(\mu_i)=\sum_{k,l} c_{k,l}^i \mu_k \otimes \mu_l$. It is easily checked that this map does not depend on the choice of a basis, as we have $a(\mu)(K,L)=\int_G \mu(K +gL)dg$.

It was shown in \cite{bernig_fu06} that the restriction of $\PD$ to $G$-invariant valuations, denoted by $\PD^G$, is an isomorphism and that the following diagram commutes:
\begin{equation} \label{eq_ftaig}
\begin{tikzcd}
    \Val^G \arrow[r, "a"] \arrow[d, "\PD^G"] & \Val^G \otimes \Val^G  \arrow[d, "\PD^G \otimes \PD^G"] \\
    \Val^{G*}  \arrow[r, "c^*"] & \Val^{G*}  \otimes \Val^{G*}                                
\end{tikzcd}
\end{equation}
where $c^*$ denotes the adjoint of the convolution product $c:\Val^G \otimes \Val^G \to \Val^G$.

The aim of this section is to prove a version of these formulas for compactly supported constructible functions. Note that we can evaluate a valuation~$\mu \in \Val^G$ on a compactly supported constructible function by Example \ref{ex_evaluate_on_constructible}.   

\begin{thm}[{see~\Cref{main:additive-kinematic-formulas}}]
Let $G \subset \mathrm O(n)$ be a closed subgroup that acts transitively on the sphere. 
The formula~\eqref{eq_additive_kinematic_convex} holds for~$\phi_1,\phi_2 \in \mathrm{CF}_\cpt(\R^n)$ with the same constants: 
\begin{equation*}
    \int_G \mu_i(\phi_1 * g_*\phi_2) dg= \sum_{k,l} c_{k,l}^i \mu_k(\phi_1) \mu_l(\phi_2).
\end{equation*}
\end{thm}

\proof 
Let $\phi_1,\phi_2 \in \mathrm{CF}_\cpt(\R^n)$ and $\mu_i \in \Val^G$. Then by $G$-invariance of~$\mu_i$,
{
\begin{align*}
\int_G &\mu_i(\phi_1 * g_*\phi_2) dg = 	\int_{G \times G} \mu_i((g_1)_*\phi_1 * (g_2)_*\phi_2) dg_1 dg_2 \\[0.5em]
& = \int_{G \times G} \left\langle \mu_i, M((g_1)_*\phi_1 * (g_2)_*\phi_2)\right\rangle dg_1dg_2\\[0.5em]
& = \int_{G \times G} \left\langle \mu_i, M((g_1)_*\phi_1) * M((g_2)_*\phi_2)\right\rangle dg_1dg_2  \quad \text{(Cor.~\ref{cor:convolution-CFgval}, Thm.~\ref{thm_compatibility_convolution_generalized},}\\[-0.5em]
    &\hspace{10cm} \text{Prop.~\ref{prop_convolution_constructible_transitive}})\\
& = \left\langle \mu_i, \int_G M((g_1)_*\phi_1)dg_1 * \int_G M((g_2)_*\phi_2) dg_2\right\rangle \\[0.5em]
    & = \left\langle c^* \circ \PD^G(\mu_i), \int_G M((g_1)_*\phi_1)dg_1 \otimes \int_G M((g_2)_*\phi_2) dg_2\right\rangle \\[0.5em]
&\hspace{-0.4em}\overset{\eqref{eq_ftaig}}{=} \left\langle (\PD^G \otimes \PD^G) \circ a(\mu_i), \int_G M((g_1)_*\phi_1)dg_1 \otimes \int_G M((g_2)_*\phi_2) dg_2\right\rangle \\[0.5em]	
& = a(\mu_i)(\phi_1 \otimes \phi_2) \quad \text{ (by $(G \times G)$-invariance of $a(\mu_i)$)}\\[0.5em]
& = \sum_{k,l} c_{k,l}^i \mu_k(\phi_1) \mu_l(\phi_2).
\end{align*}
}
\endproof

\subsection{Multiplicative kinematic formula on $S^3$}
In this section, we consider the case of the action of~$G=\mathrm{SO}(4)$ on the 3-sphere.

\begin{prop} \label{prop_transversality_sphere}
Let $\phi_1,\phi_2 \in \mathrm{CF}(S^3)$. Then for almost all $(g_1,g_2) \in G \times G$, the generalized valuations $(g_1)_* \phi_1$ and $(g_2)_*\phi_2$ satisfy the assumptions of \Cref{cor:convolution-CFgval}. 
\end{prop}

\proof
The general strategy is as in Proposition \ref{prop_transversal_compact_valuations}, with some minor adjustments.

Identifying $S^3$ with unit quaternions, the 3-sphere~$H:=S^3$ is a Lie group. We write $L_h,R_h:H \to H$ for left and right multiplication by $h$.  

Let us denote~$X=H \times H$ and~$Y=H$ and consider the multiplication map~$X \to Y$. We want to find an angular stratification of $N((g_1)_*\phi_1 \boxtimes (g_2)_*\phi_2)$ which is transversal to the submanifold~$X\times_Y\P_Y$.

Since left and right multiplications by elements of~$H$ are diffeomorphisms, we have~$X \times_Y \P_Y \subset \P_X \setminus (\mathcal M_1 \cup \mathcal M_2)$, so that it is enough to study the strata of~$N((g_1)_*\phi_1 \boxtimes (g_2)_*\phi_2)$ inside the open set~$\P_X \setminus (\mathcal M_1 \cup \mathcal M_2)$. 

To do so, we first proceed to some identifications. Using left translations, we can identify~$\P_H$ and~$H \times \P_+(\mathfrak h^*)$. Explicitly, an element~$(h,[\xi]) \in \P_H$ is mapped to~$(h,[dL_h^*\xi]) \in H \times \P_+(\mathfrak h^*)$. Moreover, using the metric, we may identify~$\P_+(\mathfrak h^*)$ and the unit sphere~$S^2 \subset \mathfrak h$, by mapping a vector~$v \in S^2$ to~$[\langle \bullet,v\rangle]$. As usual, the adjoint action of~$H$ on~$\mathfrak h$ is denoted by~$\Ad$, it preserves the scalar product and hence acts on~$S^2$.
Similarly, the bundle~$\P_{H \times H}$ can be identified with~$H \times H \times \P_+(\mathfrak h^* \oplus \mathfrak h^*)$, by mapping~$(h_1,h_2,[\xi_1:\xi_2])$ to~$(h_1,h_2,[dL_{h_1}^*\xi_1:dL_{h_2}^*\xi_2])$. Next~$\P_+(\mathfrak h^* \oplus \mathfrak h^*)$ can be identified with~$S^5$. 

Under our identifications, we have~$\P_X \setminus (\mathcal M_1 \cup \mathcal M_2)=H \times H \times S_{++}^5$ where:
\begin{displaymath}
S^5_{++}:=\left\{x \in S^5: (x_1,x_2,x_3) \neq 0, (x_4,x_5,x_6) \neq 0\right\}.
\end{displaymath}
This set can be described using the following diffeomorphism, called spherical join:
\begin{displaymath}
S^2 \times S^2 \times \left(0,\frac{\pi}{2}\right) \to S^5_{++}, \quad (v_1,v_2,\psi) \mapsto (\cos \psi v_1,\sin \psi v_2).
\end{displaymath}

Then, a stratification defining~$N(\phi_1 \boxtimes \phi_2)$ on the open set $H \times H \times S^5_{++}$ is given by the strata:
\begin{displaymath}
\hat S=\left\{(h_1,h_2,\cos \psi v_1,\sin \psi v_2) : (h_i,v_i) \in S_{i} \mbox{ for } i =1,2, \psi \in \left(0,\frac{\pi}{2}\right)\right\},
\end{displaymath}  
where~$S_{i}$ runs over all strata of angular stratifications of~$H \times S^2$ that respectively define~$N(\phi_i)$ for~$i=1,2$.

To get a stratification defining~$N((g_1)_*\phi_1 \boxtimes (g_2)_*\phi_2)$ on the open set~$H \times H \times S^5_{++}$, consider the group homomorphism~$H \times H \to G$, sending~$(e,f)$ to the map~$x \mapsto e x \bar f$. Its kernel is~$\{(1,1),(-1,-1)\}$ and~$H \times H$ is a double cover of~$G$. An element $(e,f) \in H^2$ acts on~$H \times S^2$ by~$L_eR_{\bar f} \otimes \Ad_f$, and the identification~$\P_H\cong H \times \P_+(\mathfrak h^*)$ is $(H \times H)$-equivariant for this action.
Then, considering a lift~$(e_i,f_i) \in H^2$ of~$g_i \in G$, the strata
\begin{equation} \label{eq_rotated_stratum}
    \begin{split}
        \Big\{\big(e_1h_1\bar f_1,e_2h_2\bar f_2,\cos \psi &\Ad_{f_1}v_1,\sin \psi \Ad_{f_2}v_2\big) :\\
        &(h_i,v_i) \in S_i \mbox{ for } i =1,2, \psi \in \left(0,\frac{\pi}{2}\right)\Big\}
    \end{split}
\end{equation}
define~$N((g_1)_*\phi_1 \boxtimes (g_2)_*\phi_2)$ on the open set~$H \times H \times S^5_{++}$.

Finally, we claim that the map~$\Xi:\hat S \times H^4 \to H \times H \times S^5_{++}$ sending
\begin{equation*}
    (h_1,h_2,(\cos \psi v_1,\sin \psi v_2), e_1,f_1,e_2,f_2)
\end{equation*}
to
\begin{equation*}
    (e_1h_1\bar f_1,e_2 h_2 \bar f_2, (\cos \psi \Ad_{f_1}v_1, \sin \psi \Ad_{f_2}v_2))
\end{equation*}
is transversal to $X \times_Y \P_Y$. Indeed, this map is a submersion, since the left action of $H$ on itself and the coadjoint action of $H$ on $\P_+(\mathfrak h^*)=S^2$ are both transitive. Therefore, Thom's transversality theorem \cite{GP10} ensures that for almost all $(e_1,f_1,e_2,f_2) \in H^4$, the subanalytic subset~$\Xi(\hat S,e_1,f_1,e_2,f_2)$, which is equal to the set in~\eqref{eq_rotated_stratum}, is transversal to $X \times_Y \P_Y$.

A similar argument as in the proof of~\Cref{prop_convolution_constructible_transitive} shows that the stratification given by~\eqref{eq_rotated_stratum} is an angular stratification.
\endproof

Since $G$ acts transitively on the unit sphere bundle of $S^3$, the space~$\mathcal{V}(S^3)^G$ of $G$-invariant smooth valuations on~$S^3$ is finite-dimensional. It is well-known that it is spanned by the Crofton valuations $\nu_0,\nu_1,\nu_2,\nu_3$.

Let~$c:\mathcal{V}(S^3)^G \otimes \mathcal{V}(S^3)^G \to \mathcal{V}(S^3)^G$ be the convolution product on~$S^3$. More precisely, for~$\zeta_1,\zeta_2 \in \mathcal{V}(S^3)^G$ the pushforward of the exterior product~$\zeta_1 \boxtimes \zeta_2 \in \mathcal{V}^{\infty}(S^3 \times S^3)$ under the multiplication map~$S^3 \times S^3 \to S^3$ is again a smooth valuation by \cite[Thm.~1]{AB17} that will be denoted by~$c(\zeta_1,\zeta_2)=\zeta_1 * \zeta_2$. 

Let us denote by $\PD^G$ the composition of the natural maps
\begin{displaymath}
\mathcal V(S^3)^G \hookrightarrow \mathcal V^\infty(S^3) \stackrel{\PD}{\hookrightarrow} \mathcal V^{-\infty}(S^3) \twoheadrightarrow (\mathcal V(S^3)^G)^*.
\end{displaymath}

This map is an isomorphism of vector spaces, moreover $(\PD^G)^*=\PD^G$, see \cite[Definition 2.14 and Corollary 2.18]{bernig_fu_solanes}.

\begin{prop} \label{prop_ftaig_s3}
There is a map $m:\mathcal{V}(S^3)^G \to \mathcal{V}(S^3)^G \otimes \mathcal{V}(S^3)^G$ such that for $\phi_1,\phi_2 \in \mathrm{CF}(S^3)$
    \begin{equation} \label{eq_kinematic operator}
        m(\mu)(\phi_1,\phi_2)=\int_G \mu(\phi_1 * g_*\phi_2) dg,
    \end{equation}	
    and the following diagram commutes
\begin{equation} \label{eq_ftaig_s3}
        \begin{tikzcd}
        \mathcal{V}(S^3)^G \arrow[r, "m"] \arrow[d, "\PD^G"] & \mathcal{V}(S^3)^G \otimes \mathcal{V}(S^3)^G  \arrow[d, "\PD^G \otimes \PD^G"] \\
        \mathcal{V}(S^3)^{G*} \arrow[r, "c^*"] & \mathcal{V}(S^3)^{G*} \otimes \mathcal{V}(S^3)^{G*}                                      
    \end{tikzcd}.
\end{equation}
\end{prop}

\proof
Since $\PD^G$ is an isomorphism, we may define $m$ by \eqref{eq_ftaig_s3}. Let us check that \eqref{eq_kinematic operator} is satisfied. 

Let $\phi_1,\phi_2 \in \mathrm{CF}(S^3)$ and $\mu \in \mathcal V(S^3)^G$. The generalized valuations $\zeta_i:=\int_G [g_*\phi_i] dg$ are $G$-invariant, hence smooth, i.e. $\zeta_i \in \mathcal V(S^3)^G$. Since $\mu$ and $\zeta_i$ are $G$-invariant we have  
\begin{equation} \label{eq_value_pdg}
\langle \PD^G(\mu),\zeta_i\rangle=\langle \PD(\mu),\zeta_i\rangle=\left\langle \mu, \int_G [g_*\phi_i]dg\right\rangle=\int_G \mu(g_*\phi_i)dg=\mu(\phi_i).
\end{equation}

We now compute:
\begin{align*}
\int_G \mu(\phi_1 * &g_*\phi_2) dg = \int_{G \times G} \mu((g_1)_*\phi_1 * (g_2)_*\phi_2) dg_1 dg_2 \quad \text{($\mu$ is left-invariant)}\\
& = \int_{G \times G} \langle \mu, [(g_1)_*\phi_1 * (g_2)_*\phi_2]\rangle dg_1 dg_2\\
    & = \int_{G \times G} \langle \mu, [(g_1)_*\phi_1] * [(g_2)_*\phi_2]\rangle dg_1 dg_2 \quad \text{(Cor.~\ref{cor:convolution-CFgval}, Prop.~\ref{prop_transversality_sphere})}\\
& = \left\langle \mu, \int_G [(g_1)_*\phi_1] dg_1 * \int_G [(g_2)_*\phi_2] dg_2\right\rangle\\
        & = \langle \PD^G(\mu), \zeta_1 * \zeta_2\rangle \quad \text{($\mu$ and $\zeta_1 * \zeta_2$ are $G$-invariant)}\\
    & = \langle c^* \circ \PD^G(\mu), \zeta_1 \otimes \zeta_2\rangle \\
    & = \langle (\PD^G \otimes \PD^G) \circ m(\mu), \zeta_1 \otimes \zeta_2\rangle \\
& = m(\mu)(\phi_1,\phi_2) \quad \text{by \eqref{eq_value_pdg}},
\end{align*}
which proves \eqref{eq_kinematic operator}. 
\endproof

There are two ways to obtain the explicit formulas given in Theorem \ref{thm_additive_s3}. The algebra of invariant valuations on $S^3$ with respect to the convolution product was computed in \cite{bernig_faifman_kotrbaty}, which can be translated into kinematic formulas by using \Cref{prop_ftaig_s3}. Here we follow a more direct approach.   

\proof[Proof of Theorem \ref{thm_additive_s3}]
Let $f_i(r):=\nu_i(B(1,r))$, where $B(1,r)$ is a closed geodesic ball of radius $r<\frac{\pi}{2}$ and center the identity $1 \in S^3$. Then $f_i(r)$ is the normalized volume of the $r$-tube around a totally geodesic submanifold of dimension $3-i$, and using normal coordinates around the submanifold, one finds that 
\begin{align*}
f_0(r) & =1,\\
f_1(r) & =\frac{4}{\pi} \int_0^r \cos^2(t)dt=\frac{2(\cos(r)\sin(r)+r)}{\pi},\\
f_2(r) & =2 \int_0^r \sin(t) \cos(t) dt=\sin^2(r),\\
f_3(r) & =\frac{2}{\pi} \int_0^r \sin(t)^2 dt=\frac{r-\cos(r)\sin(r)}{\pi}. 
\end{align*} 

By~\Cref{prop_ftaig_s3} we know that there are constants $d^i_{k,l}$ such that 
\begin{displaymath}
\int_G \nu_i(\phi_1 * g_*\phi_2)dg=\sum d^i_{k,l} \nu_k(\phi_1)\nu_l(\phi_2), \quad \phi_1,\phi_2 \in \mathrm{CF}(S^3).
\end{displaymath}
We take $\phi_1=\1_{B(1,r)}, \phi_2=\1_{B(1,s)}$ with $r+s< \frac{\pi}{2}$. Then for every $g\in G$ we have~$\phi_1* g_*\phi_2=\1_{B(g,r+s)}$. We thus must have 
\begin{displaymath}
f_i(r+s)=\sum d_{k,l}^i f_k(r)f_l(s).
\end{displaymath}

Solving this system of equations gives us the formulas stated in the theorem.  
\endproof

\bibliographystyle{plain}
\bibliography{bibliography}

\end{document}

%% file: preambule.tex
\usepackage{fontenc}
\usepackage[utf8]{inputenc}
\usepackage{amssymb, amsmath, amsthm}
\usepackage[colorlinks = true, urlcolor = blue, linkcolor=red, citecolor=blue]{hyperref}
\usepackage{cleveref}
\usepackage{graphicx}
\usepackage{tikz-cd}
\usepackage{enumitem} 
\usepackage{stmaryrd} 
\usepackage{mathtools}

\usepackage{xparse} 

\usepackage[textsize=tiny]{todonotes}

\numberwithin{equation}{section}

\setlist[itemize]{label=\tiny\textbullet}
\setlist[enumerate]{label=(\roman*)}

\theoremstyle{plain}
\newtheorem{thm}{Theorem}[section]
\newtheorem*{thm*}{Theorem}

\newtheorem{lem}[thm]{Lemma}
\newtheorem{cor}[thm]{Corollary}
\newtheorem{prop}[thm]{Proposition}

\theoremstyle{definition}
\newtheorem{defi}[thm]{Definition}

\newtheorem{ex}[thm]{Example}
\newtheorem{rk}[thm]{Remark}

\theoremstyle{remark}

\Crefname{lem}{Lemma}{Lemmas}
\Crefname{thm}{Theorem}{Theorems}

\def\Z{{\mathbb Z}}    
\def\R{{\mathbb R}}    
\def\C{{\mathbb C}}    
\def\P{\mathbb{P}}

\def\d{\,\mathrm{d}}

\DeclareMathOperator{\id}{id}
\DeclareMathOperator{\supp}{supp}

\DeclareMathOperator{\Val}{Val}
\DeclareMathOperator{\Ad}{Ad}
\DeclareMathOperator{\Geod}{Geod}

\newcommand{\rightarrowdbl}{\rightarrow\mathrel{\mkern-14mu}\rightarrow}

\renewcommand{\tilde}[1]{\widetilde{#1}}
\makeatletter
\newcommand{\doublewidetilde}[1]{{%
  \mathpalette\double@widetilde{#1}%
}}
\newcommand{\double@widetilde}[2]{%
  \sbox\z@{$\m@th#1\widetilde{#2}$}%
  \ht\z@=.8\ht\z@
  \widetilde{\box\z@}%
}
\makeatother

\renewcommand{\bar}[1]{\overline{#1}}
\renewcommand{\hat}[1]{\widehat{#1}}
\let\oldphi\phi 
\renewcommand{\phi}{\varphi} 
\newcommand{\dualdot}[1]{\left\langle#1\right\rangle}
\DeclareMathOperator{\Int}{Int}
\newcommand{\closure}[1]{\overline{#1}}

\def\base{X} 
\def\pt{\mathrm{pt}}

\def\dEuler{\d\chi}
\def\cpt{\mathrm{c}}

\def\1{{\mathbf{1}}}    

\NewDocumentCommand{\CF}{O{} O{\base}}{%
\mathrm{CF}_{#1}(#2)%
}

\DeclareMathOperator{\CC}{CC}

\newcommand{\basecur}[1]{C_{#1}}

\NewDocumentCommand{\emCFgval}{d[]}{
\IfNoValueTF{#1}{\def\argument{[-]}}{\def\argument{\left[#1\right]}}%
\argument%
}

\def\PD{\mathrm{PD}}
\def\del{\partial}
\NewDocumentCommand{\lfc}{O{T^*X} O{n}}{
C^\mathrm{BM}_{#2}(#1)%
}
\NewDocumentCommand{\BM}{O{T^*X} O{n}}{
H^\mathrm{BM}_{#2}(#1)%
}
\def\scurBMnotation{\Psi}
\NewDocumentCommand{\scurBM}{O{M} O{}}{
\scurBMnotation_{#1}^{#2}%
}

\newcommand{\CLag}[1]{\mathcal{CL}(#1)}
\newcommand{\Lag}[1]{\mathcal{L}(#1)}

\NewDocumentCommand{\forms}{O{\base} O{} m}{%
\Omega_{#2}^{#3}(#1)%
}

\def\zero{\underline{0}}

\newcommand{\submfd}[1][\base]{\mathcal{P}(#1)}

\def\valsymb{\mathcal{V}}
\NewDocumentCommand{\val}{O{} O{\base}}{%
\valsymb_{#1}^{\infty}(#2)%
}
\NewDocumentCommand{\gval}{O{\base}}{%
\valsymb^{-\infty}(#1)%
}

\def\lettercurspace{\mathcal{D}}
\NewDocumentCommand{\cur}{m O{\base}}{%
\lettercurspace_{#1}(#2)
}
\newcommand{\intcur}[1]{\llbracket#1\rrbracket}

\NewDocumentCommand{\scur}{m O{\base}}{%
\lettercurspace^{\mathrm{sub}}_{#1}\!\left(#2\right)%
}

\DeclareMathOperator{\WF}{WF}
\DeclareMathOperator{\Dens}{Dens}
\DeclareMathOperator{\vol}{vol}

\DeclareMathOperator{\cone}{cone}
